\documentclass{amsart}

\usepackage{amsmath}
\usepackage{amssymb}
\usepackage{commath}
\usepackage{mathtools}
\usepackage[all]{xy}

\usepackage{bigints}

\usepackage{hyperref}

\allowdisplaybreaks

\theoremstyle{plain}
\newtheorem{theorem}{Theorem}
\newtheorem{lemma}[theorem]{Lemma}
\newtheorem{proposition}[theorem]{Proposition}
\newtheorem{corollary}[theorem]{Corollary}
\theoremstyle{definition}
\newtheorem{definition}[theorem]{Definition}
\newtheorem{remark}[theorem]{Remark}

\numberwithin{theorem}{subsection}
\numberwithin{equation}{subsection}

\newcommand{\C}{\mathbb{C}}
\newcommand{\HH}{\mathbb{H}}

\newcommand{\R}{\mathbb{R}}

\newcommand{\Z}{\mathbb{Z}}

\newcommand{\GL}{\mathrm{GL}}

\newcommand{\Ad}{\mathrm{Ad}}
\newcommand{\ad}{\mathrm{ad}}
\newcommand{\VSO}{\mathrm{VSO}}

\newcommand{\im}{\operatorname{Im}}
\newcommand{\re}{\operatorname{Re}}

\newcommand{\fg}{\mathfrak{g}}
\newcommand{\fh}{\mathfrak{h}}

\newcommand{\gl}{\mathfrak{gl}}

\newcommand{\cS}{\mathcal{S}}
\newcommand{\cT}{\mathcal{T}}
\newcommand{\cA}{\mathcal{A}}

\newcommand{\cF}{\mathcal{F}}
\newcommand{\mfF}{\mathfrak{F}}

\begin{document}

\title[Heisenberg on Siegel and Bergman spaces]{The Heisenberg group action on the Siegel domain and the structure of Bergman spaces}

\author{Julio A. Barrera-Reyes}
\address{Centro de Investigaci\'on en Matem\'aticas, Guanajuato, M\'exico}
\email{julio.barrera@cimat.mx}

\author{Ra\'ul Quiroga-Barranco}
\address{Centro de Investigaci\'on en Matem\'aticas, Guanajuato, M\'exico}
\email{quiroga@cimat.mx}



\begin{abstract}
	We study the biholomorphic action of the Heisenberg group $\mathbb{H}_n$ on the Siegel domain $D_{n+1}$ ($n \geq 1$). Such $\mathbb{H}_n$-action allows us to obtain decompositions of both $D_{n+1}$ and the weighted Bergman spaces $\mathcal{A}^2_\lambda(D_{n+1})$ ($\lambda > -1$). Through the use of symplectic geometry we construct a natural set of coordinates for $D_{n+1}$ adapted to $\mathbb{H}_n$. This yields a useful decomposition of the domain $D_{n+1}$. The latter is then used to compute a decomposition of the Bergman spaces $\mathcal{A}^2_\lambda(D_{n+1})$ ($\lambda > -1$) as direct integrals of Fock spaces. This effectively shows the existence of an interplay between Bergman spaces and Fock spaces through the Heisenberg group $\mathbb{H}_n$. As an application, we consider $\mathcal{T}^{(\lambda)}(L^\infty(D_{n+1})^{\mathbb{H}_n})$ the $C^*$-algebra acting on the weighted Bergman space $\mathcal{A}^2_\lambda(D_{n+1})$ ($\lambda > -1$) generated by Toeplitz operators whose symbols belong to $L^\infty(D_{n+1})^{\mathbb{H}_n}$ (essentially bounded and $\mathbb{H}_n$-invariant). We prove that $\mathcal{T}^{(\lambda)}(L^\infty(D_{n+1})^{\mathbb{H}_n})$ is commutative and isomorphic to $\mathrm{VSO}(\mathbb{R}_+)$ (very slowly oscillating functions on $\mathbb{R}_+$), for every $\lambda > -1$ and $n \geq 1$. 
\end{abstract}

\maketitle


\section{Introduction} 
\label{sec:intro}
Analytic function spaces and their operators is a current and very active research topic. Among the classical examples of these we have the weighted Bergman spaces over the Siegel domain, the natural unbounded realization of the complex unit ball. Equally important are the Toeplitz operators acting on such Bergman spaces. The two of them together yield concrete examples of Hilbert spaces and $C^*$-algebras of operators acting on them. In fact, it has been very fruitful to study these objects towards a better understanding of both complex analysis and operator theory.

A Toeplitz operator comes from the choice of a (usually essentially bounded) measurable function called a symbol. Choosing a family of such symbols then yields a $C^*$-algebra: the one generated by the corresponding Toeplitz operators. To understand the structure of these $C^*$-algebras in terms of the choice of symbols is at the core of the study of Toeplitz operators. For example, on strictly pseudoconvex bounded domains with smooth boundary the $C^*$-algebra generated by Toeplitz operators whose symbols are continuous up to the boundary is essentially commutative (see~\cite[Theorem~4.1.25]{UpmeierToepBook}). 

It is by now well known that geometric restrictions on the families of symbols yield Toeplitz operators that generate actually commutative $C^*$-algebras. On the unit disk, it was proved in \cite{KorenblumZhu1995} that radial symbols (depending only on the radial variable of polar coordinates) yield commutative $C^*$-algebras. More recent works have shown that this phenomenon is part of a more general principle: on bounded symmetric domains, families of symbols invariant under suitable biholomorphism groups yield Toeplitz operators that generate commutative $C^*$-algebras. Beyond the seminal work \cite{KorenblumZhu1995} this was proved on the upper half-plane for dilations and translations (see \cite{GKVHyperbolic,GKVParabolic}) and for the so-called maximal Abelian subgroups of biholomorphisms on the unit ball and the Siegel domain (see \cite{QVUnitBall1}). Later on, a large family of biholomorphism groups acting on bounded symmetric domains that yield commutative $C^*$-algebras through invariant symbols were exhibited in \cite{DOQJFA}. This line of work is also related to the understanding of the properties of commutators and semicommutators of Toeplitz operators. For this we refer to \cite{AppuhamyLe2016,AxlerCuckovicRao2000,ChoeKooLee2004,CuckovicLouhichi2008,Le2017} as examples.

Within the scope of this setup, we consider for $n \geq 1$ the $n+1$-dimensional Siegel domain $D_{n+1}$ and the biholomorphic action of the Heisenberg group $\HH_n$ on $D_{n+1}$. We study the symbols on $D_{n+1}$ that are $\HH_n$-invariant and the $C^*$-algebras generated by the corresponding Toeplitz operators on the weighted Bergman spaces. We will denote such $C^*$-algebras by $\cT^{(\lambda)}(L^\infty(D_{n+1})^{\HH_n})$, where $\lambda > -1$. We refer to Definition~\ref{def:two-types-symbols} and subsection~\ref{subsec:Toeplitz-HHn-invariant} where this notation is introduced.

With respect to the Toeplitz operators with $\HH_n$-invariant symbols we obtain in Theorem~\ref{thm:Heisenberg-Toeplitz-multiplier} simultaneous diagonalizing expressions given by multiplier operators acting on $L^2(\R_+)$. In particular we have a proof, by explicit diagonalization, of the commutativity of the $C^*$-algebras $\cT^{(\lambda)}(L^\infty(D_{n+1})^{\HH_n})$, for every $\lambda > -1$ and $n \geq 1$.

A striking fact of the diagonalizing formulas obtained in Theorem~\ref{thm:Heisenberg-Toeplitz-multiplier} is that, although they depend on the weight $\lambda$ as expected, they do not depend on $n$ and so they are the same for all Siegel domains $D_{n+1}$ where $n \geq 1$. We state this fact in Theorem~\ref{thm:Heisenberg-Toeplitz-C*-commutative} which yields another of our main results: for all $\lambda > -1$, the $C^*$-algebras $\cT^{(\lambda)}(L^\infty(D_{n+1})^{\HH_n})$ are all isomorphic, independently of $n$, to a fixed $C^*$-algebra acting on $L^2(\R_+)$. 

As for the weight $\lambda > -1$, we prove that the isomorphism class of the $C^*$-algebras $\cT^{(\lambda)}(L^\infty(D_{n+1})^{\HH_n})$ is independent of $\lambda$ as well. Our Theorem~\ref{thm:Toeplitz-Hn-VSO} proves that the $C^*$-algebras $\cT^{(\lambda)}(L^\infty(D_{n+1})^{\HH_n})$ are all isomorphic to $\VSO(\R_+)$, the algebra of very slowly oscillating functions on $\R_+$ (see Definition~\ref{def:VSO}), acting on $L^2(\R_+)$. The reason is that the diagonalizing formulas \eqref{eq:gamma-for-Heisenberg} obtained in Theorem~\ref{thm:Heisenberg-Toeplitz-multiplier} yield exactly the same family of functions found in \cite{HHM2014VerticalWeighted,HMV2013Vertical} for Toeplitz operators with vertical symbols acting on the weighted Bergman spaces over the upper half-plane. Since the $C^*$-algebras generated by the latter sort of Toeplitz operators have been proved in \cite{HHM2014VerticalWeighted} to be isomorphic to $\VSO(\R_+)$ acting on $L^2(\R_+)$, we thus conclude that $\cT^{(\lambda)}(L^\infty(D_{n+1})^{\HH_n})$ have such isomorphism class as well for every $\lambda > -1$ and $n \geq 1$. For further details on this reasoning and its relevance we refer to Remarks~\ref{rmk:Heisenberg-Toeplitz-C*-commutative} and \ref{rmk:Toeplitz-Hn-VSO}. The latter prove that the most natural generalization of the horizontal translations on the upper half-plane is the action of the Heisenberg group $\HH_n$ on $D_{n+1}$; this is so from the viewpoint of Toeplitz operators with suitably invariant symbols and the $C^*$-algebras that they generate.

To obtain the theorems described above we prove in Theorem~\ref{thm:Rlambda-Ulambda} a result which is interesting on its own: the existence, for every $\lambda > -1$ and $n \geq 1$, of a unitary map
\[
	U_\lambda : 
	\cA^2_\lambda(D_{n+1}) \longrightarrow
		\int_{\R_+}^\oplus \cF^2_{2\xi}(\C^n) \dif \xi,
\]
that decomposes the weighted Bergman spaces $\cA^2_\lambda(D_{n+1})$ as a direct integral of Fock spaces $\cF^2_{2\xi}(\C^n)$, for $\xi \in \R_+$. This decomposition is the consequence of our construction of a Segal-Bargmann type transform on the $L^2$-spaces containing the weighted Bergman spaces. This sort of construction usually requires a careful choice of coordinates on the domain in question. For this we resort to Lie theory and symplectic geometry, which are natural tools given the action of the Heisenberg group.

We consider in Section~\ref{sec:sympgeom-Siegel} the notion of moment maps for symplectic actions on K\"ahler manifolds and compute the moment map for the action of the Heisenberg group $\HH_n$ on $D_{n+1}$, as well as for some of its subgroups which we classify for the connected case. Next, we use this to define moment map symbols (see Definition~\ref{def:moment-map-symbol} and \cite{QSMomentMapJFA}). It is proved in Proposition~\ref{prop:Hn-orbits-moment-map-center} that the moment map symbols corresponding to the center of $\HH_n$ are precisely the $\HH_n$-invariant symbols. We define in subsection~\ref{subsec:group-moment-coordinates} a set of coordinates on $D_{n+1}$ that are very well adapted to the action of $\HH_n$. Furthermore, these coordinates are defined using both the action of $\HH_n$ and the moment map of the center of $\HH_n$. They turn out to be quite natural since they greatly simplify the construction of our Segal-Bargmann type transform. For example, we need to take Fourier transform on a single variable only: the central coordinate of $\HH_n$. As a comparison the Segal-Bargmann transform obtained in \cite{QVUnitBall1} for the nilpotent MASG (maximal Abelian subgroup) required to take Fourier transform in as many coordinates as the dimension of the Siegel domain. Also to compare previous techniques with ours we show in subsection~\ref{subsec:nilpotent-MASG} how to obtain the diagonalizing formulas \eqref{eq:gamma-for-Heisenberg} using the nilpotent MASG. This is possible due to the fact that for symbols $\HH_n$-invariance implies invariance under the nilpotent MASG (see Corollary~\ref{cor:Heisenberg-inv-nilpMASG-inv}). As noted in Remark~\ref{rmk:Heisenberg-vs-nilpotentMASG} our methods are more simple and clear than going through the nilpotent MASG. Most importantly, our Lie theoretic and symplectic methods provide more information and insight of the Toeplitz operators with $\HH_n$-invariant symbols.

\section{Analysis on the Siegel domain}
\label{sec:Siegel}
\subsection{Bergman spaces and Toeplitz operators}
In the rest of this work we will consider the complex vector space $\C^{n+1}$ and assume that $n \geq 1$. Furthermore, any vector $z \in \C^{n+1}$ will be decomposed as $z = (z',z_{n+1})$ where $z' \in \C^n$ and $z_{n+1} \in \C$. The Siegel domain in $\C^{n+1}$ is defined~by
\[
	D_{n+1} = \{ z \in \C^{n+1} \mid \im(z_{n+1}) > |z'|^2 \},
\]
which is well known to be biholomorphically equivalent to the $(n+1)$-dimensional unit ball through the Cayley transform. In particular, $D_{n+1}$ is a symmetric domain that yields the unbounded realization of the $(n+1)$-dimensional unit ball.

We will denote by $\dif z$ the Lebesgue measure on $\C^{n+1}$, and for every $\lambda > -1$ we will consider the measure $v_\lambda$ on $D_{n+1}$ given by 
\[
	\dif v_\lambda(z) 
		= \frac{c_\lambda}{4} \big(\im(z_{n+1}) - |z'|^2\big)^\lambda \dif z,
\]
where the constant $c_\lambda$ is defined as
\[
	c_\lambda = \frac{\Gamma(\lambda + n + 2)}{\pi^{n+1} \Gamma(\lambda + 1)}.
\]
The weighted Bergman space on the Siegel domain $D_{n+1}$ with weight $\lambda > -1$ is denoted by $\cA^2_\lambda(D_{n+1})$ and consists of the holomorphic functions on $D_{n+1}$ that belong to $L^2(D_{n+1}, v_\lambda)$. It is well known that $\cA^2_\lambda(D_{n+1})$ is a closed subspace of $L^2(D_{n+1},v_\lambda)$ and that its associated orthogonal projection $B_{n+1,\lambda} : L^2(D_{n+1},v_\lambda) \rightarrow \cA^2_\lambda(D_{n+1})$ satisfies
\[
	B_{n+1,\lambda}(f)(z) 
			= \int_{D_{n+1}} f(w) K_{n+1,\lambda}(z,w) \dif v_\lambda(w),
\]
where the function $K_{n+1,\lambda} : D_{n+1} \times D_{n+1} \rightarrow \C$ is given by 
\[
	K_{n+1,\lambda}(z,w) 
		= \frac{1}{\Big(\displaystyle\frac{z_{n+1} - \overline{w}_{n+1}}{2i} - z' \cdot \overline{w'}\Big)^{\lambda + n + 2}}.					
\]
The projection $B_{n+1,\lambda}$ and the function $K_{n+1,\lambda}$ are called the Bergman projection and the Bergman kernel of $D_{n+1}$, respectively.

For every $a \in L^\infty(D_{n+1})$, we will denote by $T^{(\lambda)}_a = T_a$ the Toeplitz operator acting on the Bergman space $\cA^2_\lambda(D_{n+1})$, which is defined as the compression of the multiplier operator operator $M_a$. In other words, we have 
\[
	T^{(\lambda)}_a = B_{n+1,\lambda} M_a|_{\cA^2_\lambda(D_{n+1})}.
\]

\subsection{The Heisenberg group}
In this section we describe our main Lie theory tools: the Heisenberg group and its action on the Siegel domain. This can be considered as a particular case of the quasi-translations studied in \cite[Page~27]{UpmeierToepBook}. Here we state some properties that 
will be used latter on, we prove some of them for the sake of completeness and refer to \cite{UpmeierToepBook} for further details. We note that our Siegel domain may be seen as a generalized upper half-space, while \cite[Example~1.3.72]{UpmeierToepBook} considers what may be described as a left-hand side half-space. Nevertheless, a simply linear biholomorphism relates both realizations.

Let us consider in $\C^n \times \R$ the product given by
\[
	(z',s) \cdot (w', t) 
		= (z'+ w', s + t + 2 \im(z'\cdot \overline{w'})),
\]
where $z',w' \in \C^n$ and $s,t \in \R$. It is straightforward to check that $\C^n \times \R$ endowed with this operation becomes a Lie group, which we will denote by $\HH_n$ and will be called the \textbf{Heisenberg group}. Note that $(0,0)$ is the identity element and that we also have $(z',s)^{-1} = (-z',-s)$ for this product.

The next result collects the very basic properties of the Heisenberg group.

\begin{proposition}\label{prop:Heisenberg-properties}
	For every $n \geq 1$, the Heisenberg group $\HH_n$ has Lie algebra given by $\fh_n = \C^n \times \R$ carrying the Lie brackets
	\[
		[(w',t),(z',s)] = (0, 4\im(w'\cdot \overline{z'})).
	\]
	The exponential map $\exp : \fh_n \rightarrow \HH_n$ is the identity map, and the adjoint representation $\Ad_{\HH_n} : \HH_n \rightarrow \GL(\fh_n)$ is given by
	\[
		\Ad_{\HH_n}(w',t)(z',s) = (z', s + 4 \im(w'\cdot \overline{z'})).
	\]
	Furthermore, the centers $Z(\HH_n)$ and $Z(\fh_n)$ of $\HH_n$ and $\fh_n$ are both given by $\{0\} \times \R \subset \C^n \times \R$, considered as a subgroup and as a subspace, respectively.
\end{proposition}
\begin{proof}
	Since $\HH_n$ has the manifold structure given by a real vector space, its tangent space at the identity $(0,0)$ is $\C^n \times \R$. Hence, the latter yields the underlying vector space of the Lie algebra $\fh_n$. 
	
	From the definition of the product in $\HH_n$ it follows immediately that for every fixed $(z',s) \in \HH_n$, the map $\R \rightarrow \HH_n$ given by $r \mapsto (rz',rs)$ is a homomorphism of Lie groups. By definition of the exponential map of Lie groups, it follows that the exponential map of $\HH_n$ is indeed the identity map.
	
	On the other hand, for any $(z',s), (w',t) \in \HH_n$ we have
	\[
		(rw',rt) \cdot (rz',rs) 
			=  (r(w' + z'), r(t + s) + 2r^2\im(w'\cdot \overline{z'}))
	\]
	for every $r \in \R$. Hence, it follows from our computation of the exponential map and \cite[Chapter~II~Lemma~1.8]{Helgason} or the Campbell-Baker-Hausdorff formula (see~\cite[Page~669]{KnappBeyond2nd}) that the Lie brackets in $\fh_n$ are given by
	\[
		[(w',t),(z',s)] = (0, 4\im(w'\cdot \overline{z'})).
	\]
	In particular, the Lie algebra adjoint representation $\ad : \fh_n \rightarrow \gl(\fh_n)$ satisfies $\ad(X)^2 = 0$ for every $X \in \fh_n$.
	
	The previous remarks allow us to compute the Lie group adjoint representation $\Ad_{\HH_n} : \HH_n \rightarrow \GL(\fh_n)$ as follows
	\begin{align*}	
		\Ad_{\HH_n}(w',t) &= \Ad_{\HH_n}(\exp(w',t)) = e^{\ad(w',t)} \\
			&= I_{\fh_n} + \ad(w',t)
			= I_{\fh_n} + [(w',t), \cdot],
	\end{align*}
	where the second identity follows from elementary Lie theory. This yields the required formula for $\Ad_{\HH_n}$.
		
	Finally, the claims on the centers of $\HH_n$ and $\fh_n$ are clear from the formulas obtained.
\end{proof}

We will consider the $\HH_n$-action on $\C^{n+1}$ obtained from the next elementary result, whose proof we present for the sake of completeness.

\begin{proposition}\label{prop:Hn-action}
	The assignment $\HH_n \times \C^{n+1} \rightarrow \C^{n+1}$ defined by
	\[
		(w',t)\cdot z = (z' + w', z_{n+1} + t + 2i z'\cdot \overline{w'} + i |w'|^2).
	\]
	is a free $\HH_n$-action which is holomorphic and preserves the map on $\C^{n+1}$ given by $z \mapsto \im(z_{n+1}) - |z'|^2$. In particular, this $\HH_n$-action yields a holomorphic action on~$D_{n+1}$. 
\end{proposition}
\begin{proof}
	A straightforward computation using the product law in $\HH_n$ shows that the mapping in the statements yields indeed an action (see also \cite[Example~1.3.72]{UpmeierToepBook}). It is also clear that such action yields biholomorphisms of $\C^{n+1}$.
	
	On the other hand, we have for every $(w',t) \in \HH_n$ and $z \in \C^{n+1}$
	\begin{align*}
		\im\big(\big((w',t)\cdot z\big)_{n+1}\big) 
				&- \big|\big((w',t)\cdot z\big)'\big|^2 = \\
				&= \im(z_{n+1}) + 2 \im(iz'\cdot \overline{w'}) + |w'|^2 
					- |z' + w'|^2 \\
				&= \im(z_{n+1}) + 2 \re(z'\cdot \overline{w'}) + |w'|^2
					- |z' + w'|^2 \\
				&= \im(z_{n+1}) - |z'|^2,
	\end{align*}
	thus proving the $\HH_n$-invariance of the given function. In particular, the $\HH_n$-action preserves $D_{n+1}$ and so defines a biholomorphic action on it.
\end{proof}

\begin{remark}\label{rmk:Hn-NilpotentAction}
	It is clear that the subset $\R^{n+1} = \R^n \times \R \subset \HH_n$ is a closed Lie subgroup. Moreover, the corresponding $\R^{n+1}$-action on $D_{n+1}$ yields precisely one of the maximal Abelian subgroups (MASGs) of biholomorphisms studied in \cite{QVUnitBall1,QVUnitBall2}, namely the so-called nilpotent MASG.
\end{remark}

\section{Symplectic geometry on the Siegel domain}
\label{sec:sympgeom-Siegel}
\subsection{Symplectic manifolds and moment maps}
Our main geometric tool will be symplectic geometry and the computation of moment maps of symplectic actions. Hence, we review some definitions and properties. Further details can be found in \cite{QSMomentMapJFA} and the references therein.

A symplectic manifold is a manifold $M$ carrying a non-degenerate closed $2$-form $\omega$. In this case, $\omega$ is called the symplectic form of $M$, and the pair $(M,\omega)$ will be used to denote the symplectic manifold when emphasis on the $2$-form is required. The geometric symmetries of a symplectic manifold are given by the so-called symplectomorphisms. These are diffeomorphisms $\varphi : M \rightarrow M$ that preserve $\omega$. More precisely, we require that $\omega_z(u,v) = \omega_{\varphi(z)}(\dif\varphi_z(u), \dif\varphi_z(v))$, for every $z \in M$ and $u, v \in T_z M$.

For a given smooth action of a Lie group $H$ on a manifold $M$ we have a natural map from $\fh$, the Lie algebra of $H$, to the Lie algebra of smooth vector fields on $M$. More precisely, for any $X \in \fh$ we will denote by $X^\sharp$ the vector field given by
\[
	X^\sharp_z = \frac{\dif}{\dif r}\sVert[2]_{r=0} \exp(rX)\cdot z,
\]
for every $z \in M$. It is easily seen that this map defines an anti-homomorphism of Lie algebras (see \cite[Chapter~II~Theorem~3.4]{Helgason}).

If $H$ is a Lie group acting smoothly on a symplectic manifold $M$, then we will say that the $H$-action is symplectic when every element of $H$ yields a symplectomorphism of $M$. When such an action is given we have the fundamental notion of moment map.

\begin{definition}\label{def:moment-map}
	Let $H$ be a Lie group with Lie algebra $\fh$ and denote by $\fh^*$ the dual vector space of the latter. For a symplectic $H$-action on a symplectic manifold $(M,\omega)$, a moment map for the $H$-action is a smooth map $\mu : M \rightarrow \fh^*$ for which the following properties hold.
	\begin{enumerate}
		\item For every $X \in \fh$, the function $\mu_X : M \rightarrow \R$ given by $\mu_X(z) = \langle \mu(z), X\rangle$ satisfies
		\[
			\dif \mu_X = \omega(X^\sharp, \cdot).
		\]
		\item The map $\mu : M \rightarrow \fh^*$ is $H$-equivariant for the dual action of the adjoint representation on the target. In other words, we have
		\[
			\mu(h\cdot z) = \Ad_H^*(h)(\mu(z)),
		\]
		for every $z \in M$ and $h \in H$, where $\Ad_H^*(h) = \Ad_H(h^{-1})^\top$ is the dual map of $\Ad_H(h^{-1})$ and $\Ad_H : H \rightarrow \GL(\fh)$ denotes the adjoint representation of~$H$.
	\end{enumerate}
\end{definition}

\begin{remark}\label{rmk:moment-map}
	With the notation from Definition~\ref{def:moment-map}, if the group $H$ is Abelian, then the adjoint representation is trivial and the $H$-equivariance condition reduces to an $H$-invariance condition: $\mu(h\cdot z) = \mu(z)$ for every $z \in M$ and $h \in H$. But, for a general Lie group $H$ it is important to consider both the adjoint representation and its dual representation. On the other hand, to simplify computations it is useful to introduce some sort of identification between $\fh$ and $\fh^*$. When $\fh$ is considered just as a vector space, this is straightforward: any inner product on $\fh$ implements one such identification. However, one has to carry over the adjoint representation and so not every inner product will be adequate. The first alternative is to look for an inner product in $\fh$, positive definite or not, which is $\Ad_H(H)$-invariant. Then, the corresponding identification between $\fh$ and $\fh^*$ also realizes an identification between $\Ad_H$ and $\Ad_H^*$. This allows to replace in Definition~\ref{def:moment-map} $\fh^*$ and $\Ad_H^*$ with $\fh$ and $\Ad_H$, respectively. Nevertheless, there are Lie groups $H$ for which there does not exist $\Ad_H(H)$-invariant inner products on $\fh$: the Heisenberg group is one of them. In this situation, the best alternative is to choose some ``natural'' inner product, identify $\fh$ and $\fh^*$ and compute the representation on $\fh$ that corresponds to $\Ad_H^*$ on $\fh^*$. We will follow this alternative for the Heisenberg group and its subgroups. To achieve this we consider the next general result.
\end{remark}

\begin{lemma}\label{lem:rho-from-Ad*}
	Let $H$ be a Lie group with Lie algebra $\fh$, let $\langle\cdot,\cdot\rangle$ be an inner product on $\fh$ and consider the identification $\fh \simeq \fh^*$ given by $\langle\cdot,\cdot\rangle$. Then, the map $\rho : H \rightarrow \GL(\fh)$ defined by $\langle\rho(h)(X), Y\rangle = \langle X, \Ad_H(h^{-1})(Y)\rangle$, for every $h \in H$, $X,Y \in \fh$, where $\Ad_H$ is the adjoint representation of $H$, is the representation corresponding to $\Ad_H^* : H \rightarrow \GL(\fh^*)$ for the identification $\fh \simeq \fh^*$. Furthermore, if $\langle\cdot,\cdot\rangle$ is $\Ad_H(H)$-invariant, then $\rho = \Ad_H$.
\end{lemma}
\begin{proof}
	If we consider the transpose operation with respect to the given inner product as well as our identification $\fh \simeq \fh^*$, then we clearly have $\rho(h) = \Ad_H(h^{-1})^\top$, for every $h \in H$. From this it follows that $\rho$ is a representation of $H$ on $\fh$. Furthermore, the identification $\fh \simeq \fh^*$ is given by the map $X \mapsto \langle X, \cdot \rangle$, and so the definition of $\rho$ is equivalent to the commutativity of the following diagram
	\[
		\xymatrix{
			\fh \ar[rr]^{\rho(h)} \ar[d] & & \fh \ar[d] \\	
			\fh^* \ar[rr]^{\Ad_H^*(h)}& & \fh^*
		}
	\]
	for every $h \in H$. This shows that $\rho$ corresponds to $\Ad_H^*$ under the identification $\fh \simeq \fh^*$ under consideration. The last claim is now a clear consequence.
\end{proof}

\begin{remark}\label{rmk:rho-from-Ad*}
	In the rest of this work, and in the notation of Lemma~\ref{lem:rho-from-Ad*}, we will say that $\rho$ is the representation of $H$ on its Lie algebra $\fh$ transpose to $\Ad_H$. For these notation and construction to be considered we assume the choice of some inner product in $\fh$. The upshot of Lemma~\ref{lem:rho-from-Ad*} is that $\rho$, the transpose of $\Ad_H$, is the same as $\Ad_H^*$, the dual of $\Ad_H$, up to the identification $\fh \simeq \fh^*$ given by the chosen inner product.
\end{remark}

Definition~\ref{def:moment-map} and Lemma~\ref{lem:rho-from-Ad*} yield the next consequence. We will use it as alternative definition of moment map in the rest of this work, since it will provide a more convenient way to compute moment maps for the symplectic actions that we consider.

\begin{corollary}\label{cor:moment-map-fh}
	Let $H$ be a Lie group acting by symplectomorphisms on a symplectic manifold $(M, \omega)$ and let us consider the identification $\fh \simeq \fh^*$ obtained from a given inner product $\langle\cdot,\cdot\rangle$ on $\fh$. With respect to such identification, a map $\mu : M \rightarrow \fh$ is a moment map for the $H$-action on $M$ if and only if the following conditions are satisfied.
	\begin{enumerate}
		\item For every $X \in \fh$, the function $\mu_X : M \rightarrow \R$ given by $\mu_X(z) = \langle \mu(z), X\rangle$ satisfies $\dif \mu_X = \omega(X^\sharp, \cdot)$.
		\item The map $\mu : M \rightarrow \fh$ is $H$-equivariant with respect to the representation $\rho : H \rightarrow \GL(\fh)$ transpose to $\Ad_H$ and defined in Lemma~\ref{lem:rho-from-Ad*}.
	\end{enumerate}
\end{corollary}
\begin{proof}
	We note that in condition (1) from Definition~\ref{def:moment-map} the expression $\mu_X(z) = \langle \mu(z), X\rangle$ considers the angled brackets as evaluation of a linear functional on a vector. Under the identification $\fh \simeq \fh^*$ given by the inner product such evaluation corresponds to the inner product itself. Hence, both conditions (1), from our statement and from Definition~\ref{def:moment-map}, are equivalent under the identification $\fh \simeq \fh^*$. Similarly, the equivariance required in conditions (2) from the definition and this corollary are equivalent by Lemma~\ref{lem:rho-from-Ad*} and the commutative diagram found in its~proof.
\end{proof}

It is also useful to relate the moment maps for actions of subgroups to those of the action of the ambient group. We consider this problem within the setup of the previous results.

\begin{proposition}\label{prop:moment-map-subgroups}
	Let $H$ be a Lie group acting by symplectomorphisms on a symplectic manifold $(M,\omega)$ and let $G \subset H$ be a Lie subgroup with their Lie algebras denoted by $\fg$ and $\fh$, respectively. Let us also fix an inner product in $\fh$ and the corresponding identification $\fh \simeq \fh^*$. With the notation from Corollary~\ref{cor:moment-map-fh}, if $\mu : M \rightarrow \fh$ is a moment map for the $H$-action on $M$ and $\pi : \fh \rightarrow \fg$ is the orthogonal projection, then $\pi \circ \mu : M \rightarrow \fg$ is a moment map for the $G$-action on $M$.
\end{proposition}
\begin{proof}
	We will prove that $\pi \circ \mu$ satisfies the properties stated in Corollary~\ref{cor:moment-map-fh} that characterize a moment map for the $G$-action. Hence, we will use the notation of that result.
	
	We note that for $X \in \fg$
	\[
		\big(\pi\circ\mu\big)_X (z)
			= \langle\pi\circ\mu(z),X\rangle 
			= \langle\mu(z),X\rangle 
			= \mu_X(z),
	\]
	for every $z \in M$. Hence, the map $\pi \circ \mu$ satisfies Corollary~\ref{cor:moment-map-fh}(1) for the $G$-action as a consequence of $\mu$ satisfying the same property.
	
	Let us now consider the map $\rho_H : H \rightarrow \GL(\fh)$ given by $\rho_H(h) = \Ad_H(h^{-1})^\top$, which is the representation transpose to $\Ad_H$. We will use the notation where the group associated to a representation is indicated as a subindex. We claim that $\rho_H(h)(\fg^\perp) = \fg^\perp$ for every $h \in G$. This follows from the fact that, for every $h \in G$, we have  $\Ad_G(h^{-1}) = \Ad_H(h^{-1})|_{\fg}$, while the former leaves invariant $\fg$. As a consequence we have
	\[
		\pi \circ \rho_H(h) = \pi \circ \rho_H(h) \circ \pi
	\]
	for every $h \in G$. This identity is obtained by noting that it holds for elements of both $\fg$ and $\fg^\perp$. We thus conclude that for $h \in G$ and $X,Y \in \fg$ we have
	\begin{align*}
		\langle\pi\circ\rho_H(h)\circ\pi(X),Y\rangle 
			&= \langle\Ad_H(h^{-1})^\top(X), Y \rangle
			= \langle X, \Ad_H(h^{-1})(Y) \rangle \\
			&= \langle X, \Ad_G(h^{-1})(Y) \rangle 
			= \langle \Ad_G(h^{-1})^\top(X), Y \rangle,
	\end{align*}
	which yields the identity
	\[
		\rho_G(h) = \Ad_G(h^{-1})^\top 
				= \pi \circ \rho_H(h) \circ \pi |_{\fg},
	\]
	for every $h \in G$, where $\rho_{G}$ is the representation transpose to $\Ad_G$. Once, we have proved these properties we can compute as follows for every $z \in M$ and~$h \in G$
	\begin{align*}
		(\pi\circ\mu)(hz) &= \pi(\rho_H(h)(\mu(z))) \\
				&= (\pi \circ \rho_H(h) \circ \pi) (\pi\circ\mu(z)) \\
				&= \rho_{G}(h)(\pi\circ\mu(z)),
	\end{align*}
	where we have used the $H$-equivariance of $\mu$ in the first identity. This proves that $\pi\circ\mu$ satisfies Corollary~\ref{cor:moment-map-fh}(2) for the $G$-action. We conclude from this corollary that $\pi\circ\mu$ is a moment map for the $G$-action on $M$.
\end{proof}

\subsection{K\"ahler manifolds}
We recall here some well known facts and refer to \cite[Chapter~IX]{KNII} and \cite[Chapter~2]{Mok} for further details. Let us consider a complex manifold $M$ and let us denote by $J$ its complex structure as a tensor acting on its tangent bundle. A Riemannian metric $g$ on $M$ is called Hermitian when it satisfies
\[
	g_z(J_z u, J_z v) = g_z(u,v),
\]
for every $z \in M$ and $u, v \in T_zM$. In this case, the tensor $\omega$ defined by
\[
	\omega_z(u,v) = g_z(J_z u, v),
\]
for every $z \in M$ and $u, v \in T_zM$ as well, is a non-degenerate $2$-form. With this notation, we say that the pair $(M,g)$ is a Hermitian manifold with associated $2$-form $\omega$. If the $2$-form $\omega$ is closed, then $M$ is called a K\"ahler manifold and $\omega$ is referred as its associated K\"ahler form. In particular, every K\"ahler manifold is a symplectic manifold whose symplectic form is precisely the associated K\"ahler form.

Our main example is given by the description of the Siegel domain as a K\"ahler manifold. From now on we will freely use Wirtinger partial derivatives and their dual $1$-forms. For this and the next result we refer to \cite[Chapter~IX]{KNII}, \cite[Chapter~2]{Mok} and \cite{QSMomentMapJFA}.

\begin{proposition}\label{prop:Siegel-Kahler}
	The Siegel domain $D_{n+1}$ has a K\"ahler structure whose Riemannian metric and K\"ahler form $g$ and $\omega$, respectively, are given by the following expressions.
	\begin{align*}
		g_z =&\; \frac{1}{(\im(z_{n+1}) - |z'|^2)^2}
				\bigg(
					(\im(z_{n+1}) - |z'|^2) 
						\sum_{j=1}^n \dif z_j \otimes \dif \overline{z}_j \\
				&+ \sum_{j,k=1}^n \overline{z}_j z_k \dif z_j 
											\otimes \dif \overline{z}_k 
				+ \frac{1}{2i} \sum_{j=1}^n \big(
					\overline{z}_j \dif z_j \otimes \dif \overline{z}_{n+1}
						-z_j \dif z_{n+1} \otimes \dif \overline{z}_j
					\big) \\
				&+ \frac{1}{4} \dif z_{n+1} \otimes \dif \overline{z}_{n+1}
				\bigg) \\
		\omega_z =&\; \frac{i}{(\im(z_{n+1}) - |z'|^2)^2}
				\bigg(
					(\im(z_{n+1}) - |z'|^2) 
						\sum_{j=1}^n \dif z_j \wedge \dif \overline{z}_j \\
				&+ \sum_{j,k=1}^n \overline{z}_j z_k \dif z_j 
											\wedge \dif \overline{z}_k 
				+ \frac{1}{2i} \sum_{j=1}^n \big(
					\overline{z}_j \dif z_j \wedge \dif \overline{z}_{n+1}
						-z_j \dif z_{n+1} \wedge \dif \overline{z}_j
					\big) \\
				&+ \frac{1}{4} \dif z_{n+1} \wedge \dif \overline{z}_{n+1}
				\bigg).
	\end{align*}
	Furthermore, every biholomorphism of $D_{n+1}$ preserves both tensors and so the group of automorphisms $D_{n+1}$ yields an action which is both isometric and symplectic.
\end{proposition}

In the rest of this work we will consider $D_{n+1}$ as a K\"ahler manifold for the structure described in Proposition~\ref{prop:Siegel-Kahler}. As an immediate consequence of the latter and Proposition~\ref{prop:Hn-action} we obtain the next result.

\begin{corollary}\label{cor:Hn-symplectic-action}
	The action of the Heisenberg group $\HH_n$ on the Siegel domain $D_{n+1}$ is symplectic.
\end{corollary}

\subsection{Moment maps for the Heisenberg group and its subgroups}
We will now compute a moment map for the $\HH_n$-action on $D_{n+1}$. We start with the next simple lemma that computes the vector fields on $D_{n+1}$ induced from the Lie algebra~$\fh_n$.

\begin{lemma}\label{lem:X-sharp}
	For the $\HH_n$-action on $D_{n+1}$ and for every $X = (w',t) \in \fh_n$ we have
	\begin{align*}
		X^\sharp_z &= (w', t + 2i z'\cdot \overline{w'}) \\
				&= \sum_{j=1}^n\bigg(
					w_j \frac{\partial}{\partial z_j} 
						+ \overline{w}_j \frac{\partial}{\partial \overline{z}_j}
						\bigg) 
				+ \big(t + 2i z'\cdot \overline{w'} \big) 
						\frac{\partial}{\partial z_{n+1}}
				+ \overline{\big(t + 2i z'\cdot \overline{w'} \big)} 
						\frac{\partial}{\partial \overline{z}_{n+1}}
	\end{align*}
	for every $z \in D_{n+1}$, where the first expression belongs to $\C^{n+1} = T_z D_{n+1}$ and the second expression corresponds to the vector field as a linear combination of the Wirtinger operators.
\end{lemma}
\begin{proof}
	For the given objects, Propositions~\ref{prop:Heisenberg-properties} and \ref{prop:Hn-action} imply that for every $r \in \R$ we have
	\[
		\exp(rX)\cdot z = (z' + rw', z_{n+1} + rt + 2i r z'\cdot \overline{w'}
				+ i r^2 |w'|^2),
	\]
	whose derivative at $r = 0$ yields the first expression of the statement. On the other hand, if $Y = (f_1, \dots, f_{n+1})$ is a vector field on $D_{n+1}$ written as a $\C^{n+1}$-valued function, then $Y$ in terms of the Wirtinger operators is given by (see \cite{QSMomentMapJFA})
	\[
		Y = \sum_{j=1}^{n+1} \bigg(
			f_j \frac{\partial}{\partial z_j} 
				+ \overline{f}_j \frac{\partial}{\partial \overline{z}_j}
			\bigg),
	\]
	which yields the second expression of the statement.
\end{proof}

In the next result we obtain an expression for the representation transpose to the adjoint representation of $\HH_n$ on $\fh_n$ (see Lemma~\ref{lem:rho-from-Ad*} and Remark~\ref{rmk:rho-from-Ad*}). Recall that such transpose representation corresponds to the transpose of the adjoint representation.

\begin{lemma}\label{lem:rho-for-Heisenberg}
	Consider the natural isomorphism $\fh_n = \C^n \times \R \simeq \R^{2n+1}$ and the corresponding inner product $\langle\cdot,\cdot\rangle$ induced on $\fh_n$ from the canonical inner product in $\R^{2n+1}$. Then, in the notation of Lemma~\ref{lem:rho-from-Ad*}, the representation $\rho : \HH_n \rightarrow \GL(\fh_n)$ transpose to $\Ad_{\HH_n}$ is given by $\rho(w',t)(z',s) 
	= (z' + 4is w', s)$.
\end{lemma}
\begin{proof}
	For $X = (z',s), Y = (z_1',s_1)$ belonging to $\fh_n$ and $h = (w',t) \in \HH_n$ we have
	\begin{align*}
		\langle \rho(w',t)(z',s), Y \rangle 
			&= \langle \rho(h)(X), Y \rangle \\
			&= \langle X, \Ad_{\HH_n}(h^{-1})(Y)\rangle \\
			&= \langle (z',s), (z_1', s_1 - 4\im(w'\cdot \overline{z_1'}) \rangle \\
			&= \re(z'\cdot \overline{z_1'}) + s s_1 
				- 4 \im(s w' \cdot \overline{z_1'}) \\
			&= \re(z'\cdot \overline{z_1'} + 4i s w'\cdot \overline{z_1'})
				+ s s_1 \\
			&= \re((z' + 4i s w')\cdot \overline{z_1'})
				+ s s_1 \\
			&= \langle (z' + 4is w', s), Y\rangle,
	\end{align*}
	where we have used Proposition~\ref{prop:Heisenberg-properties} in the third identity. This yields the claimed expression.
\end{proof}

\begin{remark}\label{rmk:rho-for-Heisenberg}
	In the rest of this work, we will assume $\fh_n$ endowed with the inner product considered in Lemma~\ref{lem:rho-for-Heisenberg}. This fixes the identification between $\fh_n$ and $\fh_n^*$ required by Lemma~\ref{lem:rho-from-Ad*}, and so yields a well-defined representation transpose to $\Ad_{\HH_n}$, whose expression has been obtained in Lemma~\ref{lem:rho-for-Heisenberg}.
\end{remark}

We now obtain the main result of this section. Recall that we have agreed to use the definition of moment map provided by Corollary~\ref{cor:moment-map-fh}.

\begin{theorem}\label{thm:Heisenberg-moment-map}
	The map $\mu^{\HH_n} : D_{n+1} \rightarrow \fh_n$ given by
	\[
		\mu^{\HH_n}(z) = -\frac{(4iz', 1)}{2(\im(z_{n+1}) - |z'|^2)},
	\]
	is a moment map for the $\HH_n$-action on $D_{n+1}$.
\end{theorem}
\begin{proof}
	For simplicity we will denote $\mu = \mu^{\HH_n}$ in this proof. Following Corollary~\ref{cor:moment-map-fh} and the remarks from its proof, for every $X \in \fh_n$, the function $\mu_X : D_{n+1} \rightarrow \R$ is given by
	\[
		\mu_X(z) = \langle\mu(z),X\rangle
	\]
	where $\langle\cdot,\cdot\rangle$ denotes the inner product considered in $\fh_n$. In particular, for every $X = (w',t) \in \fh_n$ we have
	\[
		\mu_X(z) = \frac{4\im(z'\cdot \overline{w'}) - t}{2(\im(z_{n+1}) - |z'|^2)}
			= \frac{2i(z'\cdot \overline{w'} - \overline{z'}\cdot w') + t}%
					{i(z_{n+1} - \overline{z}_{n+1}) + 2z'\cdot \overline{z'}}.
	\]
	We compute the differential of this function using Wirtinger derivatives to obtain 
	\begin{align*}
		\dif \big(\mu_X\big)_z &= \frac{1}{(\im(z_{n+1}) - |z'|^2)^2} \times \\
				\times
				\Bigg(&
					\sum_{j=1}^n 
						\Big(
							2\im(z'\cdot \overline{w'}) \overline{z}_j
							- i(\im(z_{n+1}) - |z'|^2) \overline{w}_j
							- \frac{1}{2}t\overline{z}_j 
						\Big) \dif z_j  \\
					&+ \sum_{j=1}^n 
						\Big(
							2\im(z'\cdot \overline{w'}) z_j
							+ i(\im(z_{n+1}) - |z'|^2) w_j
							- \frac{1}{2}tz_j 
						\Big) \dif \overline{z}_j  \\
					&+ i \Big(\im(z'\cdot\overline{w'}) 
							- \frac{1}{4} t \Big) \dif z_{n+1}
					 - i \Big(\im(z'\cdot\overline{w'}) 
					 		- \frac{1}{4} t \Big) \dif \overline{z}_{n+1}
				\Bigg)
	\end{align*}
	for every $z \in D_{n+1}$. Also, using Proposition~\ref{prop:Siegel-Kahler} and Lemma~\ref{lem:X-sharp} we obtain the following after some manipulations
	\begin{align*}
		\omega_z&(X^\sharp_z,\cdot) = \frac{1}{(\im(z_{n+1}) - |z'|^2)^2} \times \\
			\times
			\Bigg(&
				-\Big(i(z'\cdot\overline{w'} - \overline{z'}\cdot w') 
				+ \frac{1}{2}t\Big) \sum_{j=1}^n \overline{z}_j \dif z_j 
				- i(\im(z_{n+1}) - |z'|^2) \sum_{j=1}^n \overline{w}_j \dif z_j  \\
				&-\Big(i(z'\cdot\overline{w'} - \overline{z'}\cdot w') 
				+ \frac{1}{2}t\Big) \sum_{j=1}^n z_j \dif \overline{z}_j
				+ i(\im(z_{n+1}) - |z'|^2) \sum_{j=1}^n w_j \dif \overline{z}_j  \\
				&+ i\Big(\frac{1}{2i}(z'\cdot\overline{w'} - \overline{z'}\cdot w')
						- \frac{1}{4} t \Big) \dif z_{n+1} \\
				&- i\Big(\frac{1}{2i}(z'\cdot\overline{w'} - \overline{z'}\cdot w')
						- \frac{1}{4} t \Big) \dif \overline{z}_{n+1}
			\Bigg).
	\end{align*}
	Comparing both of these expressions we easily conclude that 
	\[
		\dif \big(\mu_X\big)_z = \omega_z(X^\sharp_z,\cdot),
	\]
	for every $X \in \fh_n$ at every $z \in D_{n+1}$.
	
	It remains to prove the $\HH_n$-equivariance of $\mu$ with respect to the representation $\rho$ transpose to $\Ad_{\HH_n}$. Since the map $z \mapsto \im(z_{n+1}) - |z'|^2$ is already $\HH_n$-invariant (see Proposition~\ref{prop:Hn-action}), it is enough to prove that the map $\mu_0 : D_{n+1} \rightarrow \fh_n$ defined by $\mu_0(z) = (4iz',1)$ is $\HH_n$-equivariant. Hence, we compute for $h = (w',t) \in \HH_n$
	\begin{align*}
		\mu_0(h\cdot z) 
			&= \mu_0(z' + w', z_{n+1} + t + 2i z'\cdot \overline{w'} + i |w'|^2) \\
			&= (4i(z'+w'),1) = (4iz' + 4iw', 1) 
				= \rho(w',t)(4iz',1) \\
			&= \rho(w',t)(\mu_0(z)),
	\end{align*}
	which holds for every $z \in D_{n+1}$. This completes the proof that $\mu$ is a moment map for the $\HH_n$-action.
\end{proof}

We now obtain moment maps for the subgroups of $\HH_n$. Theorem~\ref{thm:Heisenberg-moment-map} and Proposition~\ref{prop:moment-map-subgroups} will provide the main tools to achieve this. First we consider all such subgroups.

\begin{proposition}\label{prop:Heisenberg-subgroups}
	The following is a complete list of the connected Lie subgroups of~$\HH_n$.
	\begin{enumerate}
		\item $V \times \R$, where $V \subset \C^n$ is a real subspace.
		\item $H_{V,f} = \mathrm{Graph}(f)$, the graph of a real linear functional $f : V \rightarrow \R$, where $V \subset \C^n$ is an isotropic (real) subspace for the canonical symplectic structure of $\C^n$.
	\end{enumerate}
\end{proposition}
\begin{proof}
	Recall that the canonical symplectic structure of $\C^n$ is (a constant multiple) of $(z',w') \mapsto \im(z'\cdot \overline{w'})$. Hence, the sets listed in the statement are easily seen to be connected subgroups of $\HH_n$.
	
	To prove that we have listed all connected subgroups of $\HH_n$, it is enough to classify all Lie subalgebras of $\fh_n$ and consider the groups generated by their images under the exponential map of $\HH_n$, which in this case is the identity. 

	Let $\fh \subset \fh_n$ be a Lie subalgebra, and let us denote by $V$ and $V_0$ the projections of $\fh$ onto $\C^n$ and $\R$. Recall that $\fh_n = \C^n \times \R$ as a vector space. Hence, we have $\fh \subset V \times \R$ and
	\[
		\dim V \leq \dim \fh \leq \dim V + 1.
	\]
	If $\dim \fh = \dim V + 1$, then we necessarily have $\fh = V \times \R$, which is clearly a Lie subalgebra of $\fh_n$. After applying the exponential map (the identity map) we obtain a subgroup of the form (1) from the statement. 
	
	We now assume that $\dim \fh = \dim V$. We note that $\{0\} \times \R \not\subset \fh$. Otherwise, the projection of $\fh$ onto $V$ yields an isomorphism $\fh/(\{0\} \times \R) \simeq V$, which is a contradiction. Hence, there exists $f : V \rightarrow \R$ linear such that $\fh = \mathrm{Graph}(f)$. If $V$ is not an isotropic subspace of $\C^n$, then there exist $(z',f(z')), (w',f(w')) \in \fh$ such that $\im(z'\cdot \overline{w'}) \not= 0$. By Proposition~\ref{prop:Heisenberg-properties} we have $(0,4\im(z'\cdot \overline{w'})) \in \fh$, whose last coordinate is non-zero. This implies that $\{0\} \times \R \subset \fh$, which is a contradiction. Hence, $V$ is necessarily an isotropic subspace of $\C^n$. The image of $\fh = \mathrm{Graph}(f)$ under the exponential map (the identity as noted above) is now a subgroup of the form (2) from the statement.
\end{proof}

\begin{remark}\label{rmk:Heisenberg-subgroups}
	In the list of subgroups obtained in Proposition~\ref{prop:Heisenberg-subgroups}, case (1) allows $V$ to be either isotropic or non-isotropic. However, in both cases (1) and (2), the latter alternative determines whether the subgroup is Abelian or not. More precisely, it is easy to check that for both types of subgroups considered in Proposition~\ref{prop:Heisenberg-subgroups} the subgroup is Abelian if and only if $V$ is an isotropic subspace of $\C^n$. As a particular case, $\R^n \times \R$ yields the subgroup whose action on $D_{n+1}$ corresponds to the nilpotent MASG considered in \cite{QVUnitBall1,QVUnitBall2}. See also Remark~\ref{rmk:Hn-NilpotentAction}.
\end{remark}

The next result is a consequence of Theorem~\ref{thm:Heisenberg-moment-map} and Proposition~\ref{prop:moment-map-subgroups}. Also note that the particular cases can be considered because of Proposition~\ref{prop:Heisenberg-subgroups} and are obtained from an easy computation. In the rest of this work for every integer $1 \leq \ell \leq n-1$, inducing the decomposition $\C^n = \C^{n-\ell} \times \C^\ell$, we will write every $z' \in \C^n$ as $z' = (z_{(1)}, z_{(2)})$ where $z_{(1)} \in \C^{n-\ell}$ and~$z_{(2)} \in \C^\ell$.

\begin{corollary}\label{cor:subgroups-moment-map}
	Let $H \subset \HH_n$ be a connected subgroup. If $\pi_{\fh} : \fh_n \rightarrow \fh$ is the orthogonal projection for the natural inner product of $\fh_n = \C^n \times \R$, then the map $\mu^H = \pi_{\fh} \circ \mu^{\HH_n} : D_{n+1} \rightarrow \fh$ is a moment map for the $H$-action on $D_{n+1}$. In particular, we have the following moment maps.
	\begin{enumerate}
		\item For the center $Z(\HH_n)$:
			\[
				\mu^{Z(\HH_n)}(z) = 
					-\frac{1}{2(\im(z_{n+1}) - |z'|^2)}.
			\]
		\item For $H_{\R} = \R^{n+1}$:
			\[
				\mu^{H_{\R}}(z) = 
					-\frac{(-4\im(z'), 1)}{2(\im(z_{n+1}) - |z'|^2)}.
			\]
		\item For $H_{i\R} = i\R^n \times \R$:
			\[
				\mu^{H_{i\R}}(z) = 
					-\frac{(4\re(z'), 1)}{2(\im(z_{n+1}) - |z'|^2)}.
			\]
		\item For $1 \leq \ell \leq n-1$ and $H_{\ell,\R} = \C^{n-\ell} \times \R^\ell \times \R$:
			\[
				\mu^{H_{\ell,\R}}(z) = 
					-\frac{\big(4z_{(1)}, -4\im(z_{(2)}), 1\big)}%
						{2(\im(z_{n+1}) - |z'|^2)}.
			\]
		\item For $1 \leq \ell \leq n-1$ and $H_{\ell,i\R} = \C^{n-\ell} \times i\R^\ell \times \R$:
		\[
			\mu^{H_{\ell,i\R}}(z) = 
				-\frac{\big(4z_{(1)}, 4\re(z_{(2)}), 1\big)}%
					{2(\im(z_{n+1}) - |z'|^2)}.
		\]
	\end{enumerate}
\end{corollary}

\begin{remark}\label{rmk:subgroups-moment-map}
	As a consequence of Corollary~\ref{cor:subgroups-moment-map} we observe that for every connected Lie subgroup $H \subset \HH_n$ there always exists a moment map for the $H$-action on $D_{n+1}$.
\end{remark}

\section{Toeplitz operators and moment maps}
\label{sec:Toeplitz-moment}
\subsection{Moment map symbols}
We recall and apply to our particular setup the following notion already considered in \cite{QSMomentMapJFA}.

\begin{definition}\label{def:moment-map-symbol}
	Let $H \subset \HH_n$ be a connected Lie subgroup with a moment map $\mu^H : D_{n+1} \rightarrow \fh$. A symbol $a \in L^\infty(D_{n+1})$ will be called a moment map symbol for the $H$-action or a $\mu^H$-symbol if there exists a measurable function $f$ defined on the image of $\mu^H$ such that $a = f \circ \mu^H$.
\end{definition}

\begin{remark}\label{rmk:moment-map-symbol}
	Theorem~\ref{thm:Heisenberg-moment-map} and Corollary~\ref{cor:subgroups-moment-map} yield explicit formulas for moment maps for the Heisenberg group $\HH_n$ and several of its subgroups. In the rest of this work, we will use the notation introduced in Corollary~\ref{cor:subgroups-moment-map} for the subgroups considered therein. It is interesting to note that some simple coordinates changes in the target of moment maps allow to consider equivalent characterizations of the corresponding symbols. For example, a function $a \in L^\infty(D_{n+1})$ is a $\mu^{Z(\HH_n)}$-symbol if and only if there is some essentially bounded measurable function $\widetilde{a}$ such that $a(z) = \widetilde{a}(\im(z_{n+1}) - |z'|^2)$ for almost every $z \in D_{n+1}$. Similarly, such a function $a$ is a $\mu^{H_{\ell,\R}}$-symbol if and only if there is an essentially bounded measurable function $\widetilde{a}$ such that $a(z) = \widetilde{a}\big(z_{(1)}, \im(z_{(2)}), \im(z_{n+1}) - |z'|^2\big)$ for almost every $z \in D_{n+1}$. We will freely use this simplified alternative expressions whenever it is useful while being careful to clarify any non-trivial claim.
\end{remark}

An interesting feature of moment maps is that, in some cases, they satisfy invariance conditions stronger than those required in its definition. The next result will be an important example.

\begin{proposition}\label{prop:Hn-orbits-moment-map-center}
	The moment map $\mu^{Z(\HH_n)}$ computed in Corollary~\ref{cor:subgroups-moment-map} is $\HH_n$-invariant. Furthermore, the sets of $\HH_n$-orbits in $D_{n+1}$ and the level sets of the moment map $\mu^{Z(\HH_n)}$ for the $\HH_n$-action are the same. More precisely, for every $X \in \mu^{Z(\HH_n)}(D_{n+1})$ there exists $z \in D_{n+1}$ such that
	\[
		\big(\mu^{Z(\HH_n)}\big)^{-1}(X) = \HH_n z,
	\]
	the $\HH_n$-orbit of $z$, and this exhaust the family of all $\HH_n$-orbits. In particular, for any given function $a \in L^\infty(D_{n+1})$ the following conditions are equivalent.
	\begin{enumerate}
		\item $a$ is a $\mu^{Z(\HH_n)}$-symbol.
		\item $a$ is $\HH_n$-invariant.
	\end{enumerate}
\end{proposition}
\begin{proof}
	By the formula for $\mu^{Z({\HH_n})}$ from  Corollary~\ref{cor:subgroups-moment-map} and by Proposition~\ref{prop:Hn-action} its expression is $\HH_n$-invariant, and so it follows that every level set of this function is a union of $\HH_n$-orbits and that every $\HH_n$-orbit is contained in a level set. Hence, it is enough to prove that every level set of $\mu^{Z(\HH_n)}$ consists of a single $\HH_n$-orbit.
	
	Let $z, w \in D_{n+1}$ be such that $\mu^{Z(\HH_n)}(z) = \mu^{Z(\HH_n)}(w)$, which is equivalent to
	\[
		\im(z_{n+1}) - |z'|^2 = \im(w_{n+1}) - |w'|^2.
	\]
	If we choose $(\zeta,t) \in \HH_n$ where 
	\[
		\zeta = z' - w', \quad 
			t = \re(z_{n+1}) - \re(w_{n+1}) 
					+ 2 \im(w'\cdot \overline{(z' - w')}),
	\]
	then we can compute as follows
	\begin{align*}
		(\zeta,t)\cdot w 
			=&\; \big(z', w_{n+1} 
					+ \re(z_{n+1}) - \re(w_{n+1}) 
						+ 2 \im(w'\cdot \overline{(z' - w')}) \\
				&+ 2i w'\cdot \overline{(z' - w')} 
					+ i |z' - w'|^2\big)  \\
			=&\; \big(z', \re(z_{n+1}) + i\im(w_{n+1})  
					+ 2 \im(w'\cdot \overline{z'})   \\
				&+ 2i w'\cdot \overline{z'} - 2i|w'|^2
					+ i|z'|^2 + i|w'|^2 - 2i\re(w'\cdot\overline{z'})
					\big) \\
			=&\; \big(z', \re(z_{n+1}) + i(\im(w_{n+1}) - |w'|^2)
				+ i|z'|^2 \big) \\
			=&\; \big(z', \re(z_{n+1}) + i \im (z_{n+1}) \big) 
			= z,
	\end{align*}
	where we have used in the second to last line the identity that comes from the fact that $z$ and $w$ lie in the same level set. This completes the proof of our claim on level sets and orbits. The last part of the statement is now clear.
\end{proof}

\begin{remark}\label{rmk:Hn-orbits-boundary}
	The proof of Proposition~\ref{prop:Hn-orbits-moment-map-center} allows to improve the conclusions of Proposition~\ref{prop:Hn-action}. From the former it follows that the $\HH_n$-orbits in $\C^{n+1}$ are precisely the level sets of the map $z \mapsto \im(z_{n+1}) - |z'|^2$ defined on $\C^{n+1}$. Since the $\HH_n$-action clearly has no fixed points it follows that the group $\HH_n$ can be identified with any such orbit. We also note that the $\HH_n$-orbit defined by the condition $\im(z_{n+1}) = |z'|^2$ is precisely the topological boundary of $D_{n+1}$. For this reason the Heisenberg group $\HH_n$ is sometimes thought of as the topological boundary of the Siegel domain $D_{n+1}$.
\end{remark}

Through the examples considered above, we observe that there are two types of symbols that are useful. We fix a notation for them in the next definition.

\begin{definition}\label{def:two-types-symbols}
	Let $H \subset \HH_n$ be a connected Lie subgroup. We will denote by $L^\infty(D_{n+1})^H$ the space of $H$-invariant symbols and by $L^\infty(D_{n+1})^{\mu^H}$ the space of $\mu^H$-symbols, where $\mu^H$ is a moment map for the $H$-action on $D_{n+1}$.
\end{definition}

\begin{remark}\label{rmk:two-types-symbols}
	With the notation introduced in Definition~\ref{def:two-types-symbols}, the second claim from Proposition~\ref{prop:Hn-orbits-moment-map-center} can be subsumed in the following expression
	\[
		L^\infty(D_{n+1})^{\HH_n} = L^\infty(D_{n+1})^{\mu^{Z(\HH_n)}}.
	\]
	However, it is not known to the authors whether similar formulas hold for other subgroups $H \subset \HH_n$. More precisely and within our setup, we can ask whether a given invariance condition is equivalent to being a function of a moment map (for a possible different group) and conversely. Such problems, although interesting, are beyond the scope of this work.
\end{remark}

\subsection{Group-moment coordinates for $\HH_n$}
\label{subsec:group-moment-coordinates}
We will now introduce a set of coordinates on $D_{n+1}$ obtained from the $\HH_n$-action and (a subset of) its moment map. This will turn out to be very useful in the description of Toeplitz operators with $\HH_n$-invariant symbols acting on $\cA^2_\lambda(D_{n+1})$.

From now on, we will consider the function $H : D_{n+1} \rightarrow \R_+$ defined by
\[
	H(z) = \frac{1}{\im(z_{n+1}) - |z'|^2}. 
\]
It follows from Theorem~\ref{thm:Heisenberg-moment-map} and Corollary~\ref{cor:subgroups-moment-map} that $H$ is, up to a constant, the last coordinate of the moment map for the $\HH_n$-action and the moment map of the $Z(\HH_n)$-action. Since the function $H$ has as level sets the $\HH_n$-orbits (see Proposition~\ref{prop:Hn-orbits-moment-map-center}), the values of $H$ and the $\HH_n$-orbits can be considered as complementary values that may allow to build coordinates for $D_{n+1}$. To achieve this we will introduce maps $\sigma : \R_+ \rightarrow D_{n+1}$ and $\rho : D_{n+1} \rightarrow \HH_n$, so that $\R_+$ can be seen as a coordinate submanifold of $D_{n+1}$ and so that the elements of $D_{n+1}$ have $\HH_n$-coordinates. For this to hold, we impose the conditions
\[
	H(\sigma(r)) = r, \quad z = \rho(z)\cdot \sigma(H(z)),
\]
for all $r \in \R_+$ and $z \in D_{n+1}$, in order to make $H$ and $\sigma$ compatible as coordinates and involve the $\HH_n$-action, respectively. It is easy to very that the functions defined~by
\[
	\sigma(r) = \Big(0',\frac{i}{r}\Big), \quad \rho(z) = (z',\re(z_{n+1})),
\]
satisfy the conditions required above. From now on, we will consider $\sigma$ and $\rho$ so defined. The next result shows that these functions can be put together to obtain a coordinate system for $D_{n+1}$. It turns out that the weighted measures on $D_{n+1}$ are particularly simple in these coordinates.

\begin{proposition}\label{prop:group-moment-coordinates}
	The functions given by
	\begin{align*}
		\kappa : &\;\HH_n \times \R_+ \longrightarrow D_{n+1} \\
		&\kappa(w',t,r) = (w',t)\cdot \sigma(r) 
			= \Big(w', t + \frac{i}{r} + i|w'|^2\Big) \\
		\tau : &\;D_{n+1} \longrightarrow \HH_n \times \R_+ \\
		&\tau(z) = (\rho(z), H(z)) 
			= \Big(z', \re(z_{n+1}), \frac{1}{\im(z_{n+1}) - |z'|^2}\Big),
	\end{align*}
	are smooth and inverse of each other. In particular, they are both diffeomorphisms. Furthermore, the push-forward under $\tau$ of the weighted measure $v_\lambda$ is the measure on $\HH_n \times \R_+$ given by
	\[
		\dif \nu_\lambda(w',t,r) 
			= \frac{c_\lambda}{4r^{\lambda + 2}} 
				\dif w' \dif t \dif r,
	\]
	for every $\lambda > -1$.
\end{proposition}
\begin{proof}
	A straightforward computation shows that both functions are indeed inverses of each other, and so they yield a set of coordinates. It is also easy to show using rectangular coordinates that we have
	\[
		|\det(\dif \tau_{(x,y)})| = \frac{1}{(y_{n+1} - |x'|^2 - |y'|^2)^2}
			= r^2,
	\]
	where $z' = x' + iy'$ and $z_{n+1} = x_{n+1} + i y_{n+1}$. Note that for the last identity we have used the change of coordinates corresponding to $\kappa$ and $\tau$. We conclude that the Lebesgue measure $\dif z$ on $D_{n+1}$ in the coordinates $(w',t,r)$ is given by
	\begin{align*}
		\dif z &= |\det(\dif \kappa_{(w',t,r)})| \dif w' \dif t \dif r \\
			&= |\det(\dif \tau_{\kappa(w',t,r)})|^{-1} \dif w' \dif t \dif r
			= \frac{1}{r^2} \dif w' \dif t \dif r.
	\end{align*}
	On the other hand, in terms of the coordinates $(w',t,r)$, the weight applied to the Lebesgue measure to obtain $v_\lambda$ is given by
	\[
		\frac{c_\lambda}{4} (\im(z_{n+1}) - |z'|^2)^\lambda = 
			\frac{c_\lambda}{4r^\lambda}.
	\]
	Hence, the formula in the statement is now clear.
\end{proof}

\begin{remark}\label{rmk:group-moment-coordinates}
	The coordinates on $D_{n+1}$ given by $\kappa$ and $\tau$ in Proposition~\ref{prop:group-moment-coordinates} are particularly well-suited to the $\HH_n$-action and a moment map of the $Z(\HH_n)$-action. By its very definition, the map $\kappa$ is $\HH_n$-equivariant for the action on the domain on the first factor only and the last coordinate of $\tau$ is basically a moment map for the $Z(\HH_n)$-action, the center of $\HH_n$. For this reason, in the rest of these work we will refer to the coordinates defined by $\kappa$ and $\tau$ as the \textbf{group-moment coordinates} associated to the $\HH_n$-action on $D_{n+1}$.
\end{remark}

From now on and with the notation of Proposition~\ref{prop:group-moment-coordinates}, we will denote by $U_0$ the map given by the assignment
\[
	f \longmapsto f\circ \kappa,
\]
that takes complex-valued functions defined on $D_{n+1}$ onto corresponding ones defined on $\HH_n \times \R_+$. In other words, $U_0$ is the change of coordinates, from $D_{n+1}$ to $\HH_n \times \R_+$, corresponding to the group-moment coordinates associated to $\HH_n$. We observe that this assignment does not require any further assumption on the functions involved and that its inverse $U_0^{-1}$ is given by
\[
	f \longmapsto f\circ \tau.
\]
Both maps have some important properties. First, they map smooth functions onto smooth functions. Secondly, the restriction of $U_0$ to $v_\lambda$-square-integrable functions yields a unitary map $L^2(D_{n+1}, v_\lambda) \rightarrow L^2(\HH_n \times \R_+, \nu_\lambda)$ that we will denote with the same symbol. Both of these claims are consequence of Proposition~\ref{prop:group-moment-coordinates}.

The next result computes the Cauchy-Riemann equations on $D_{n+1}$ with respect to the group-moment coordinates associated to $\HH_n$.

\begin{lemma}\label{lem:CR-U0}
	For a smooth function $\varphi : \HH_n \times \R_+ \rightarrow \C$, where $\HH_n \times \R_+$ has coordinates denoted by $(w',t,r)$, the following sets of equations are equivalent.
	\begin{enumerate}
		\item For every $j = 1, \dots, n+1$:
			\[
				U_0 \frac{\partial}{\partial \overline{z}_j} 
					U_0^{-1} \varphi = 0.
			\]
		\item For every $j = 1, \dots, n$:
			\[
				\bigg(\frac{\partial}{\partial \overline{w}_j}
					+ w_j r^2 \frac{\partial}{\partial r}\bigg) 
					\varphi = 0
			\]
			and
			\[
				\bigg(i r^2 \frac{\partial}{\partial r} 
					- \frac{\partial}{\partial t} \bigg)
					\varphi = 0.
			\]
	\end{enumerate}
\end{lemma}
\begin{proof}
	Let us consider $\HH_n \times \R_+$ as an open subset of $\C^{n+1}$ through the assignment
	\[
		(w',t,r) \longmapsto (w', t+ir).
	\]
	Then, a straightforward application of the chain rule for the Wirtinger partial derivatives allows us to obtain the following formulas
	\begin{align*}
		\bigg(\frac{\partial}{\partial \overline{z}_j} 
			U_0^{-1} \varphi\bigg)&(z) =
				\frac{\partial \varphi}{\partial \overline{w}_j}(\tau(z)) \\
				&+ \frac{iz_j}{(\im(z_{n+1})-|z'|^2)^2}
				\bigg(
					\frac{\partial \varphi}{\partial w_{n+1}}(\tau(z))
					- \frac{\partial \varphi}{\partial \overline{w}_{n+1}}
					(\tau(z))
				\bigg) \\
		\bigg(\frac{\partial}{\partial \overline{z}_{n+1}} 
				U_0^{-1} \varphi\bigg)&(z) =
				\frac{1}{2}\bigg(
					\frac{\partial \varphi}{\partial w_{n+1}}(\tau(z))
					+ \frac{\partial \varphi}{\partial
						 \overline{w}_{n+1}}(\tau(z))
				\bigg) \\
				&+ \frac{1}{2(\im(z_{n+1})-|z'|^2)^2}\bigg(
					\frac{\partial \varphi}{\partial w_{n+1}}(\tau(z))
					- \frac{\partial \varphi}{\partial
						 \overline{w}_{n+1}}(\tau(z))
				\bigg),
	\end{align*}
	where $j = 1, \dots, n$. The equivalence is now obtained applying $U_0$.	
\end{proof}

We now consider the usual Fourier transform
\[
	\mfF(f)(\xi) = \frac{1}{\sqrt{2\pi}}	
		\int_\R f(t) e^{-i\xi t} \dif t,
\]
and define the unitary map $U_1 = I \otimes \mfF \otimes I$ acting on $L^2(\HH_n \times \R_+, \nu_\lambda)$ that applies the transform $\mfF$ on the variable $t$. We recall that $\HH_n \times \R_+$ has coordinates denoted by $(w',t,r)$. Then, by the usual properties of the Fourier transform the set of equations from Lemma~\ref{lem:CR-U0} are equivalent to
\begin{align}
	\bigg(\frac{\partial}{\partial \overline{w}_j}
	+ w_j r^2 \frac{\partial}{\partial r}\bigg) 
	\varphi &= 0 \quad j=1, \dots, n, \label{eq:CR-U1U0} \\
	\bigg(r^2 \frac{\partial}{\partial r} -\xi\bigg)
	\varphi &= 0. \notag
\end{align}
We now describe the general solution of these equations.

\begin{lemma}\label{lem:CR-U1}
	A function $\varphi : \HH_n \times \R_+ \rightarrow \C$ satisfies equations \eqref{eq:CR-U1U0} if and only if there is a function $\psi : \C^n \times \R \rightarrow \C$ holomorphic in its variable $w' \in \C^n$ such that
	\[
		\varphi(w',\xi,r) = e^{-\xi|w'|^2} e^{-\frac{\xi}{r}} \psi(w',\xi),
	\]
	for every $(w',\xi,r) \in \HH_n \times \R_+$.
\end{lemma}
\begin{proof}
	If $\varphi$ has the expression in terms of $\psi$ as in the statement, then by Leibniz rule we have
	\begin{align*}
		\frac{\partial \varphi}{\partial \overline{w}_j}(w',\xi,r)
			&= -\xi w_j \varphi(w',\xi,r) \\
		\frac{\partial \varphi}{\partial r}(w',\xi,r)
			&= \frac{\xi}{r^2} \varphi(w',\xi,r),
	\end{align*}
	where $j = 1, \dots, n$. And these are clearly equivalent to equations \eqref{eq:CR-U1U0}.
	
	Conversely, let us assume that $\varphi$ is a solution to \eqref{eq:CR-U1U0} and let us consider the function $\psi : \HH_n \times \R \rightarrow \C$ given by
	\[
		\psi(w',\xi,r) = \varphi(w',\xi,r) e^{\xi|w'|^2} e^{\frac{\xi}{r}}.
	\]
	It is enough to prove that $\psi$ is independent of $r$ and holomorphic in $w'$. For this we compute
	\[
		\frac{\partial \psi}{\partial r}(w',\xi,r) = 
			\frac{\partial \varphi}{\partial r}(w',\xi,r) 
				e^{\xi|w'|^2} e^{\frac{\xi}{r}}
			- \frac{\xi}{r^2} \varphi(w',\xi,r) e^{\xi|w'|^2} 
				e^{\frac{\xi}{r}} = 0,
	\]
	which vanishes because of the last equation from \eqref{eq:CR-U1U0}. We also have
	\begin{align*}
		\frac{\partial \psi}{\partial \overline{w}_j}(w',\xi,r) 
			&= \frac{\partial \varphi}{\partial \overline{w}_j}(w',\xi,r)
					e^{\xi|w'|^2} e^{\frac{\xi}{r}}
				+ \xi w_j \varphi(w',\xi,r) e^{\xi|w'|^2} e^{\frac{\xi}{r}} \\
			&= -w_j r^2 \frac{\partial \varphi}{\partial r}(w',\xi,r)
					e^{\xi|w'|^2} e^{\frac{\xi}{r}}
				+ \xi w_j \varphi(w',\xi,r) e^{\xi|w'|^2} e^{\frac{\xi}{r}} 
				= 0,
	\end{align*}
	where we have used again equations \eqref{eq:CR-U1U0}.
\end{proof}

\begin{remark}\label{rmk:group-moment-coord}
	In this subsection we have introduced two important constructions. Firstly, the change of variable given by the group-moment coordinates. Secondly, a Fourier transform on the (last) real variable of the Heisenberg group. These lead us to obtain a characterization of holomorphic functions on $D_{n+1}$ as suitable functions on $\HH_n \times \R_+$. Lemma~\ref{lem:CR-U1} yields the main result with this respect.
\end{remark}

\subsection{Toeplitz operators with $\HH_n$-invariant symbols}
\label{subsec:Toeplitz-HHn-invariant}
We will now consider Toeplitz operators whose symbols are $\HH_n$-invariant. Recall that Proposition~\ref{prop:Hn-orbits-moment-map-center} shows that such symbols may equivalently be described as $\mu^{\Z(\HH_n)}$-symbols (see also Remark~\ref{rmk:two-types-symbols}).

Building from the results in the previous subsection, we need to consider the square-integrability of the solutions to the equations \eqref{eq:CR-U1U0} obtained in Lemma~\ref{lem:CR-U1}. To achieve this, we introduce a functional renormalization in the next result.

\begin{lemma}\label{lem:Vlambda}
	Let us consider the measure given by
	\[
		\dif \eta (w',\xi) =
		\bigg(\frac{2\xi}{\pi}\bigg)^n e^{-2\xi|w'|^2}\dif w' \dif \xi,
	\]	
	on $\C^n \times \R_+$ and the measure $\nu_\lambda$ on $\HH_n \times \R_+$ defined in Proposition~\ref{prop:group-moment-coordinates}. Then, the map
	\begin{align*}
		V_\lambda : L^2(\C^n \times \R_+, \eta)
 			&\longrightarrow L^2(\HH_n \times \R_+, \nu_\lambda) \\
		(V_\lambda\psi)(w',\xi,r) &= 
			2\sqrt{\frac{\pi(2\xi)^{\lambda+n+1}}%
				{\Gamma(\lambda + n + 2)}}
					e^{-\xi|w'|^2} e^{-\frac{\xi}{r}}
					 \chi_{\R_+}(\xi) \psi(w',\xi),
	\end{align*}
	is a well-defined isometry whose adjoint operator is given by 
	\[
		(V_\lambda^*\varphi)(w',\xi) = 	
			\frac{\sqrt{(2\xi)^{\lambda-n+1}\Gamma(\lambda+n+2)}}%
				{2\sqrt{\pi}\Gamma(\lambda + 1)}
					e^{\xi|w'|^2}
				\int_0^\infty
					\frac{\varphi(w',\xi,r)
						e^{-\frac{\xi}{r}} \dif r}{r^{\lambda+2}}.
	\]
\end{lemma}
\begin{proof}
	For a given $\psi \in L^2(\C^n \times \R_+, \eta)$ we compute as follows.
	\begin{align*}
		\|V_\lambda\psi\|^2 
			=&\; \int_{\C^n \times \R_+^2} 
				\frac{4\pi (2\xi)^{\lambda+n+1}}%
					{\Gamma(\lambda + n + 2)}
					e^{-2\xi|w'|^2} e^{-\frac{2\xi}{r}}
						|\psi(w',\xi)|^2
					\frac{c_\lambda}{4r^{\lambda + 2}} 
						\dif w' \dif \xi \dif r \\
			=&\; \frac{\pi c_\lambda}%
				{\Gamma(\lambda + n + 2)}
				\int_{\C^n \times \R_+}
					(2\xi)^{\lambda+n+1}
					e^{-2\xi|w'|^2} |\psi(w',\xi)|^2 \times \\
			&\times \bigg(\int_0^\infty
					e^{-\frac{2\xi}{r}}
					\frac{1}{r^{\lambda + 2}} \dif r
					\bigg) \dif w' \dif \xi \\
			=&\; \frac{1}{\pi^n \Gamma(\lambda + 1)}
				\int_{\C^n \times \R_+}
					(2\xi)^{\lambda+n+1}
					e^{-2\xi|w'|^2} |\psi(w',\xi)|^2 
					\frac{\Gamma(\lambda + 1)}{(2\xi)^{\lambda + 1}}
					\dif w' \dif \xi \\
			=&\; \int_0^\infty
				\bigg(
					\int_{\C^n} |\psi(w',\xi)|^2 
						\bigg(\frac{2\xi}{\pi}\bigg)^n
						e^{-2\xi|w'|^2} \dif w'
				\bigg) \dif \xi,
	\end{align*}
	and so the first claim has been proved. If we now also choose a function $\varphi \in L^2(\HH_n \times \R_+, \nu_\lambda)$, then we have
	\begin{align*}
		\langle V_\lambda \psi, &\varphi \rangle = \\
			=&\; \int_{\C^n \times \R_+^2}
				2\sqrt{
					\frac{\pi(2\xi)^{\lambda+n+1}}%
					{\Gamma(\lambda+n+2)}
					}
					e^{-\xi|w'|^2} e^{-\frac{\xi}{r}}
				\frac{\psi(w',\xi)
					\overline{\varphi(w',\xi,r)} 
					c_\lambda}{4r^{\lambda + 2}} 
					\dif w' \dif \xi \dif r \\
			=&\; \int_{\C^n \times \R_+}
				\psi(w',\xi)
				\bigg(
					\sqrt{
					\frac{\pi(2\xi)^{\lambda+n+1}}{\Gamma(\lambda+n+2)}
					} e^{\xi|w'|^2}
					\int_0^\infty
					\frac{\overline{\varphi(w',\xi,r)}
						e^{-\frac{\xi}{r}} \dif r}{r^{\lambda+2}}
				\bigg) \times \\
			&\times 
				\frac{\Gamma(\lambda + n + 2)}{2\pi^{n+1}\Gamma(\lambda + 1)} 
					e^{-2\xi|w'|^2}
					\dif w' \dif \xi \\
			=&\; \int_{\C^n \times \R_+}
				\psi(w',\xi)
					\bigg(
					\frac{\sqrt{(2\xi)^{\lambda-n+1}\Gamma(\lambda+n+2)}}%
						{2\sqrt{\pi}\Gamma(\lambda + 1)}
						 e^{\xi|w'|^2} \times \\
			&\times \int_0^\infty
					\frac{\overline{\varphi(w',\xi,r)}
						e^{-\frac{\xi}{r}} \dif r}{r^{\lambda+2}}
					\bigg) 
				\bigg(\frac{2\xi}{\pi}\bigg)^n e^{-2\xi|w'|^2}
				\dif w' \dif \xi, \\
	\end{align*}
	which yields the required expression for the adjoint operator $V_\lambda^*$.
\end{proof}

\begin{remark}\label{rmk:Vlambda}
	We observe that the map $V_\lambda$ from Lemma~\ref{lem:Vlambda} is defined using functions of the form considered in Lemma~\ref{lem:CR-U1}, except for the holomorphicity condition on the variable $w' \in \C^n$. In other words, Lemma~\ref{lem:Vlambda} and the map $V_\lambda$ take into account only the last equation from \eqref{eq:CR-U1U0}. The role of Lemma~\ref{lem:Vlambda} is to renormalize the functions with respect to the variable $\xi$ to obtain square-integrable functions and an isometric assignment from them. The next step is to consider the holomorphicity condition on the variable $w' \in \C^n$. For this we will make use of direct integrals of Hilbert spaces, specifically weighted Fock spaces. For the basic properties and definitions associated to the former we refer to \cite[Chapter~14]{KRvolII}. As for the latter, $\cF^2_\xi(\C^n)$ will denote the Fock space over $\C^n$ with weight $\xi > 0$, which is defined as the (closed) subspace of holomorphic functions belonging to $L^2(\C^n, (\xi/\pi)^n e^{-\xi|w'|^2} \dif w')$. 
\end{remark}

\begin{lemma}\label{lem:Wlambda-directint-Fock}
	The map
	\[
		W_\lambda : \int_{\R_+}^\oplus \cF^2_{2\xi}(\C^n) \dif \xi
			\longrightarrow L^2(\HH_n \times \R_+, \nu_\lambda),
	\]
	defined as the restriction of $V_\lambda$ given in Lemma~\ref{lem:Vlambda} (to the direct integral of the family $(\cF^2_{2\xi}(\C^n))_{\xi \in \R_+}$ over $\R_+$ with the Lebesgue measure) is an isometry with image $\big(U_1 U_0 \big)\big(\cA^2_\lambda(D_{n+1})\big)$.
\end{lemma}
\begin{proof}
	We observe that the subspace of functions $\psi \in L^2(\C^n \times \R_+, \eta)$ holomorphic in the variable $w' \in \C^n$ is a closed subspace and (the) direct integral of the family $(\cF^2_{2\xi}(\C^n))_{\xi \in \R_+}$ over $\R_+$ with the Lebesgue measure. To see this, we note first that for such a function $\psi$, Fubini's theorem yields
	\[
		\int_{\R_+}
			\bigg(
				\int_{\C^n} |\psi(w',\xi)|^2 
				\bigg(\frac{2\xi}{\pi}\bigg)^n
				e^{-2\xi|w'|^2} \dif w'
			\bigg) \dif \xi =
		\int_{\C^n \times \R_+} |\psi|^2 \dif \eta
			< \infty
	\]
	which implies that $\psi(\cdot,\xi) \in \cF^2_{2\xi}(\C^n)$, for almost every $\xi \in \R_+$. This same formula and elementary convergence theorems can be used to prove that the space of functions $\psi$, as described above, is closed. It is also a simple exercise to prove that such space is in fact the direct integral claimed to be.
	
	On the other hand, the (closed) subspace $\big(U_1 U_0 \big)\big(\cA^2_\lambda(D_{n+1})\big)$ of the Hilbert space $L^2(\HH_n\times \R_+, \nu_\lambda)$ consists precisely of functions on $\HH_n \times \R_+$ of the form
	\[
		(w',\xi,r) \longmapsto
			e^{-\xi|w'|^2} e^{-\frac{\xi}{r}} \psi(w',\xi)
	\]
	such that $\psi$ is holomorphic in $w' \in \C^n$ and square-integrable with respect to $\eta$. This claim is a consequence of the definitions of $U_0$ and $U_1$ as well as Lemmas~\ref{lem:CR-U0} and \ref{lem:CR-U1}. Also note that the computations from Lemma~\ref{lem:Vlambda} show that one must restrict those functions in the variable $\xi$ from $\R$ to $\R_+$ to ensure the square-integrability. The operator $V_\lambda$ from Lemma~\ref{lem:Vlambda} has considered such restriction in its definition. Hence, $V_\lambda$ yields the same family of functions since it just introduces a renormalization on the variable $\xi$. The only condition not considered by $V_\lambda$ is the holomorphicity in the variable $w' \in \C^n$, which is however achieved by $W_\lambda$. Hence we conclude that indeed the image of $W_\lambda$ is $\big(U_1 U_0 \big)\big(\cA^2_\lambda(D_{n+1})\big)$. This discussion can be summarized in the following diagram
	\[
		\xymatrix{
		L^2(D_{n+1},v_\lambda) \ar[rrr]^{U_1 U_0} 
			&&& L^2(\HH_n \times \R_+, \nu_\lambda)
			& L^2(\C^n \times \R_+, \eta) \ar[l]_{V_\lambda} \\
		\cA^2_\lambda(D_{n+1}) \ar[u] 
			\ar[rrr]^{U_1 U_0|_{\cA^2_\lambda(D_{n+1})}}
			&&& L^2(\HH_n \times \R_+, \nu_\lambda) \ar@{=}[u]
			& \displaystyle\int_{\R_+}^\oplus \cF^2_{2\xi}(\C^n) 
				\dif \xi \ar[u] \ar[l]_{W_\lambda} 
		}
	\]
	where the extreme left and right vertical arrows are inclusions and the horizontal arrows in the second row are the restriction of the corresponding ones in the first row. We have proved above that in the diagram the operators $U_1 U_0|_{\cA^2_\lambda(D_{n+1})}$ and $W_\lambda$	have the same image. Finally, $W_\lambda$ is the restriction of an isometry and so it is itself an isometry.
\end{proof}

The next result provides a description of the weighted Bergman spaces that will allow us to diagonalize Toeplitz operator with $\HH_n$-invariant symbols. 

\begin{theorem}\label{thm:Rlambda-Ulambda}
	With the operators $U_0$, $U_1$ and $W_\lambda$ considered above, let us define the operator
	\[
		R_\lambda = W_\lambda^* U_1 U_0 : 
			L^2(D_{n+1}, v_\lambda)
			\longrightarrow
			\int_{\R_+}^\oplus \cF^2_{2\xi}(\C^n) \dif \xi.			
	\]
	Then, $R_\lambda$ is a surjective partial isometry with initial space $\cA^2_\lambda(D_{n+1})$. In other words, we have
	\[
		R_\lambda R_\lambda^* = I, \quad
		R_\lambda^* R_\lambda = B_{n+1,\lambda},
	\]
	the identity on the direct integral above and the Bergman projection from $L^2(D_{n+1},v_\lambda)$ onto $\cA^2_\lambda(D_{n+1})$, respectively. In particular, the restriction
	\[
		U_\lambda = R_\lambda|_{\cA^2_\lambda(D_{n+1})} : 
			\cA^2_\lambda(D_{n+1}) \longrightarrow
			\int_{\R_+}^\oplus \cF^2_{2\xi}(\C^n) \dif \xi,
	\]
	is a unitary map.
\end{theorem}
\begin{proof}
	From Lemma~\ref{lem:Wlambda-directint-Fock} we know that $W_\lambda$ is an isometry and so it is a partial isometry with initial space its domain and final space its image. Hence, $W_\lambda^*$ is a surjective partial isometry and $R_\lambda$ is a surjective partial isometry as well. Note that the latter claim follows from the fact that $U_0$ and $U_1$ are unitary maps. The surjectivity of $R_\lambda$ shows that its final space is the whole target, and this implies that $R_\lambda R_\lambda^* = I$, the identity on the direct integral from the statement.
	
	Let us now consider the following maps
	\[
		\xymatrix{
			L^2(D_{n+1},v_\lambda) \ar[rr]^{U_1 U_0} 
				& &
			L^2(\HH_n \times \R_+,\nu_\lambda) 
			\ar@/_/[rr]_{W_\lambda^*}
				& & 
			\displaystyle\int_{\R_+}^\oplus \cF^2_{2\xi}(\C^n) \dif \xi
			\ar@/_/[ll]_{W_\lambda} 
		}.
	\]
	From this diagram and Lemma~\ref{lem:Wlambda-directint-Fock} we conclude that $R_\lambda$ restricted to $\cA^2_\lambda(D_{n+1})$ is a unitary map onto the target direct integral. This implies that the initial space of $R_\lambda$ as a partial isometry is precisely $\cA^2_\lambda(D_{n+1})$, and so we conclude that $R_\lambda^* R_\lambda = B_{n+1,\lambda}$, the projection onto such initial space. In particular, $U_\lambda$ is a unitary map.
\end{proof}

We now obtain the diagonalization of Toeplitz operators with $\HH_n$-invariant symbols using the partial isometries constructed above from the group-moment coordinates.

\begin{theorem}\label{thm:Heisenberg-Toeplitz-multiplier}
	Let us consider $a \in L^\infty(D_{n+1})$ an $\HH_n$-invariant symbol and let $\widetilde{a} \in L^\infty(\R_+)$ be a function such that $a(z) = \widetilde{a}(\im(z_{n+1}) - |z'|^2)$, for almost every $z \in D_{n+1}$. Then, for the unitary map $U_\lambda$ defined in Theorem~\ref{thm:Rlambda-Ulambda} we have
	\[
		U_\lambda T^{(\lambda)}_a U_\lambda^* = 	M_{\gamma_{\widetilde{a},\lambda}},
	\]
	where $M_{\gamma_{\widetilde{a},\lambda}}$ is the multiplier operator acting on the direct integral of the family of Fock spaces $(\cF^2_{2\xi}(\C^n))_{\xi \in \R_+}$ over $\R_+$ with the Lebesgue measure and  $\gamma_{\widetilde{a},\lambda} \in L^\infty(\R_+)$ is the function given by
	\begin{equation}\label{eq:gamma-for-Heisenberg}
		\gamma_{\widetilde{a},\lambda}(\xi) 
			= \frac{(2\xi)^{\lambda+1}}{\Gamma(\lambda +1)}
				\int_0^\infty\widetilde{a}(r) e^{-2\xi r} r^\lambda \dif r
			= \frac{\xi^{\lambda+1}}{\Gamma(\lambda +1)}
				\int_0^\infty\widetilde{a}\bigg(\frac{s}{2}\bigg) e^{-\xi s} s^\lambda \dif s,	
	\end{equation}
	for almost every $\xi \in \R_+$.
\end{theorem}
\begin{proof}
	From the properties established in Theorem~\ref{thm:Rlambda-Ulambda} we obtain
	\begin{align*}
		U_\lambda T^{(\lambda)}_a U_\lambda^* 
			&= R_\lambda B_{n+1,\lambda} M_a B_{n+1,\lambda} R_\lambda^* 
			= R_\lambda R_\lambda^* R_\lambda
				M_a R_\lambda^* R_\lambda R_\lambda^* \\
			&= R_\lambda M_a R_\lambda^* 
			= W_\lambda^* U_1 U_0 M_a U_0^* U_1^* W_\lambda.
	\end{align*}
	We observe that the definition of $U_0$ yields
	\[
		U_0 M_a U_0^* = M_{a \circ \kappa},
	\]
	acting on $L^2(\HH_n \times \R_+, \nu_\lambda)$ and we can compute
	\[
		a \circ \kappa(w',t,r) = a((w',t)\cdot \sigma(r))
			= a(\sigma(r)) 
			= a(0',ir^{-1})
			= \widetilde{a}(r^{-1}),
	\]
	for almost every $(w',t,r) \in \HH_n \times \R_+$. We have used the $\HH_n$-invariance of $a$ in the second identity. Note that the latter function does not depend on the variables $(w',t) \in \HH_n$ and so the multiplier operator $M_{a \circ \kappa} = M_{\widetilde{a}}$ commutes with $U_1 = I \otimes \mfF \otimes I$. The reason is that the Fourier transform is applied in the variable $t$ to obtain a new variable that we have denoted by $\xi$. Hence, we arrive to
	\[
		U_\lambda T^{(\lambda)}_a U_\lambda^* 
		= W_\lambda^* M_{\widetilde{a}} W_\lambda,
	\]
	that we now proceed to compute using Lemmas~\ref{lem:Vlambda} and \ref{lem:Wlambda-directint-Fock}. For any given $\psi$ in the direct integral of the family $(\cF^2_{2\xi}(\C^n))_{\xi \in \R_+}$ over $\R_+$ with the Lebesgue measure we have for every $(w',\xi) \in \C^n \times \R_+$
	\begin{align*}
		(W_\lambda^* &M_{\widetilde{a}} W_\lambda)\psi(w',\xi) = \\
			=&\;
				\frac{\sqrt{(2\xi)^{\lambda-n+1}\Gamma(\lambda+n+2)}}%
				{2\sqrt{\pi}\Gamma(\lambda + 1)}
					e^{\xi|w'|^2} \times \\
			&\times		\int_{\R_+} 
				\widetilde{a}(r^{-1})
				2\sqrt{\frac{\pi(2\xi)^{\lambda+n+1}}%
					{\Gamma(\lambda + n + 2)}}
						e^{-\xi|w'|^2} e^{-\frac{\xi}{r}}
						\chi_{\R_+}(\xi) \psi(w',\xi)
						\frac{e^{-\frac{\xi}{r}} \dif r}{r^{\lambda+2}} \\
			=&\;
				\frac{(2\xi)^{\lambda + 1}}{\Gamma(\lambda+1)}
				\bigg(
				\int_{\R_+}
					\frac{\widetilde{a}(r^{-1}) e^{-\frac{2\xi}{r}}}%
							{r^{\lambda+2}} \dif r
				\bigg) \psi(w',\xi),
	\end{align*}
	and the result now follows from some simple changes of variable.
\end{proof}

From now on, if $\cS \subset L^\infty(D_{n+1})$ is a family of symbols, then we will denote by $\cT^{(\lambda)}(\cS)$ the $C^*$-algebra generated by Toeplitz operators, acting on $\cA^2_\lambda(D_{n+1})$ where $\lambda > -1$, with symbols belonging to $\cS$. Let us now apply  Theorem~\ref{thm:Heisenberg-Toeplitz-multiplier} to describe the structure of the $C^*$-algebra generated by Toeplitz operators with $\HH_n$-invariant symbols.

\begin{theorem}\label{thm:Heisenberg-Toeplitz-C*-commutative}
	The $C^*$-algebra $\cT^{(\lambda)}\big(L^\infty(D_{n+1})^{\HH_n}\big)$ is commutative. Furthermore, there exists a unitary map 
	\[
		U_\lambda : \cA^2_\lambda(D_{n+1}) \longrightarrow 
			\int_{\R_+}^\oplus \cF^2_{2\xi}(\C^n) \dif \xi,
	\]
	that satisfies
	\[
		U_\lambda 
			\{ T^{(\lambda)}_a : a \in L^\infty(D_{n+1})^{\HH_n} \}
		U_\lambda^* 
		= \{ M_{\gamma_{\widetilde{a},\lambda}} : \widetilde{a} 
			\in L^\infty(\R_+) \},
	\]
	where $\gamma_{\widetilde{a},\lambda}$ is defined by equation~\eqref{eq:gamma-for-Heisenberg}. Hence, $\cT^{(\lambda)}\big(L^\infty(D_{n+1})^{\HH_n}\big)$ is isomorphic to the $C^*$-subalgebra of $L^\infty(\R_+)$ generated by the functions $\gamma_{\widetilde{a},\lambda}$ given by equation~\eqref{eq:gamma-for-Heisenberg}, where $\widetilde{a} \in L^\infty(\R_+)$. In particular, the isomorphism class of $\cT^{(\lambda)}\big(L^\infty(D_{n+1})^{\HH_n}\big)$ does not depend on the dimension of the Siegel domain~$D_{n+1}$.
\end{theorem}
\begin{proof}
	The relationship between the Toeplitz operators operators $T^{(\lambda)}_a$, with $a$ an $\HH_n$-invariant symbol, and the multiplier operators $M_{\gamma_{\widetilde{a},\lambda}}$, with $\gamma_{\widetilde{a},\lambda}$ as in equation~\eqref{eq:gamma-for-Heisenberg}, has already been established in Theorem~\ref{thm:Heisenberg-Toeplitz-multiplier} for the unitary map $U_\lambda$ from its statement. From this it follows immediately that $\cT^{(\lambda)}\big(L^\infty(D_{n+1})^{\HH_n}\big)$ is commutative.
	
	For the same map $U_\lambda$, this also proves that $\cT^{(\lambda)}\big(L^\infty(D_{n+1})^{\HH_n}\big)$ corresponds to a $C^*$-subalgebra of the von Neumann algebra of diagonalizable operators for the direct integral decomposition in the statement. Such von Neumann algebra is canonically isomorphic to $L^\infty(\R_+)$ through the map $\gamma \mapsto M_\gamma$ since the direct integral is taken over $\R_+$ (see \cite[Remark~14.1.7]{KRvolII}). Within this setup, Theorem~\ref{thm:Heisenberg-Toeplitz-multiplier} implies that $\cT^{(\lambda)}\big(L^\infty(D_{n+1})^{\HH_n}\big)$ is isomorphic to the $C^*$-subalgebra of $L^\infty(\R_+)$ generated by the functions $\gamma_{\widetilde{a},\lambda}$ is defined by equation~\eqref{eq:gamma-for-Heisenberg}. This completes the proof of the theorem.
\end{proof}

\begin{remark}\label{rmk:Heisenberg-Toeplitz-C*-commutative}
	We recall that our results are stated for $D_{n+1}$ and that we have assumed $n \geq 1$, so that the lowest dimensional Siegel domain considered up to this point is $D_2$ with the action of the ($3$-dimensional) Heisenberg group $\HH_1$. Hence, we still have to consider $D_1 = \{z \in \C : \im(z) > 0\}$, the upper half-plane in the complex plane, with the action of $\R$ by the horizontal translations $z \mapsto z+t$, defined for $t \in \R$. This $1$-dimensional case has already been studied in \cite{GKVParabolic,HMV2013Vertical,HHM2014VerticalWeighted}, listed in chronological order. Theorem~\ref{thm:Heisenberg-Toeplitz-multiplier}, and a straightforward comparison with these references (see for example \cite[Theorem~2.5]{GKVParabolic}), shows that the exact same family of functions $\gamma_{\widetilde{a},\lambda}$ given by equation~\eqref{eq:gamma-for-Heisenberg} is obtained for all dimensions, including the $1$-dimensional case of the $\R$-action on $D_1$.
	
	The previous observations lead to a number of important conclusions. Firstly, from the viewpoint of $C^*$-algebras generated by Toeplitz operators, the most natural generalization of the $\R$-action on the upper half-plane $D_1$ is the action of the Heisenberg group $\HH_n$ on $D_{n+1}$. The reason is that in all such cases, and by considering Toeplitz operators with invariant symbols, we generate the same $C^*$-algebra up to isomorphism. 
	
	With this respect, a similar situation was observed in \cite{BHVRadial2014,GMVRadial} for the case of the $C^*$-algebra (weighted case and weightless case, respectively) generated by radial Toeplitz operators acting on the $n$-dimensional ball. It was proved on those references that such $C^*$-algebra is independent of $n$. However, in our case we have proved that not only is the $C^*$-algebra $\cT^{(\lambda)}\big(L^\infty(D_{n+1})^{\HH_n}\big)$ independent of $n$, up to isomorphism, but we have also proved that their natural generators, the Toeplitz operators which correspond to the multiplier operators given by \eqref{eq:gamma-for-Heisenberg}, do not actually depend on $n$ either. To the best of our knowledge, such result has not been achieved for radial Toeplitz operators on the unit ball.
	
	Secondly, thanks to the functions obtained in Theorem~\ref{thm:Heisenberg-Toeplitz-multiplier}, the currently known structure of the $C^*$-algebras involved for the $1$-dimensional case (see \cite{HHM2014VerticalWeighted}) allows us to obtain results for arbitrary dimension. We are able to provide below a precise description of the $C^*$-algebra $\cT^{(\lambda)}\big(L^\infty(D_{n+1})^{\HH_n}\big)$ for every $n \geq 1$.
\end{remark}

We recall the notion of continuity that allow us to describe the $C^*$-algebras obtained from $\HH_n$-invariant symbols.

\begin{definition}\label{def:VSO}
	A function $f : \R_+ \rightarrow \C$ is called very slowly oscillating if it is uniformly continuous for the logarithmic metric $d(x,y) = |\log(x) - \log(y)|$ on its domain $\R_+$. The set of all very slowly oscillating functions is denoted by $\VSO(\R_+)$.
\end{definition}

It is a well known fact that $\VSO(\R_+)$ is $C^*$-subalgebra of $C_b(\R_+)$. It turns out that the $C^*$-algebra $\VSO(\R_+)$ yields the description of the $C^*$-algebras generated by Toeplitz operators with $\HH_n$-invariant symbols, independently of $n$ or the weight considered.

\begin{theorem}\label{thm:Toeplitz-Hn-VSO}
	For every $n \geq 1$ and for every weight $\lambda > -1$, the $C^*$-algebra $\cT^{(\lambda)}\big(L^\infty(D_{n+1})^{\HH_n}\big)$ is isomorphic to $\VSO(\R_+)$.
\end{theorem}
\begin{proof}
	First, Theorem~\ref{thm:Heisenberg-Toeplitz-multiplier} implies that $\cT^{(\lambda)}\big(L^\infty(D_{n+1})^{\HH_n}\big)$ is isomorphic to the $C^*$-subalgebra of $L^\infty(\R_+)$ generated by the family of functions
	\[
		\mathbb{F}_\lambda =
		\{ \gamma_{\widetilde{a},\lambda} : 
				\widetilde{a} \in L^\infty(\R_+) \},
	\]
	where $\gamma_{\widetilde{a},\lambda}$ is given by equation~\eqref{eq:gamma-for-Heisenberg}. Secondly, \cite[Theorem~2]{HHM2014VerticalWeighted} proves that the same set $\mathbb{F}_\lambda$ of functions is dense in $\VSO(\R_+)$. The result follows from these two claims.
\end{proof}

\begin{remark}\label{rmk:Toeplitz-Hn-VSO}
	Building from the observations of Remark~\ref{rmk:Heisenberg-Toeplitz-C*-commutative}, we note that the proof of Theorem~\ref{thm:Toeplitz-Hn-VSO} implies the following interesting, almost canonical, properties satisfied by the family $\mathbb{F}_\lambda$ of functions $\gamma_{\widetilde{a},\lambda}$ defined by \eqref{eq:gamma-for-Heisenberg}.
	\begin{enumerate}
		\item $\mathbb{F}_\lambda$ generates $\VSO(\R_+)$ for every $\lambda > -1$. This is due to \cite[Theorem~2]{HHM2014VerticalWeighted}.
		\item $\mathbb{F}_\lambda$ correspond to the natural generators of $\cT^{(\lambda)}\big(L^\infty(D_{n+1})^{\HH_n}\big)$: the Toeplitz operators acting on $\cA^2_\lambda(D_{n+1})$ with $\HH_n$-invariant symbols. This follows from Theorem~\ref{thm:Heisenberg-Toeplitz-multiplier}.
	\end{enumerate}
	In particular, not only has Theorem~\ref{thm:Toeplitz-Hn-VSO} established that $\cT^{(\lambda)}\big(L^\infty(D_{n+1})^{\HH_n}\big)$ and $\VSO(\R_+)$ are isomorphic for all $n \geq 1$ and $\lambda > -1$. It has also proved that they have corresponding generating sets, one of which is naturally associated to the former.
\end{remark}

\subsection{A proof using the nilpotent MASG}
\label{subsec:nilpotent-MASG}
Let us consider the Abelian subgroup $\R^{n+1} \subset \HH_n$, which corresponds to the so-called nilpotent MASG (see Remark~\ref{rmk:Hn-NilpotentAction}). Then, we have the following obvious result.

\begin{corollary}\label{cor:Heisenberg-inv-nilpMASG-inv}
	For every $n \geq 1$ we have $L^\infty(D_{n+1})^{\HH_n} \subset L^\infty(D_{n+1})^{\R^{n+1}}$. In other words, every $\HH_n$-invariant symbol is symbol invariant under the nilpotent MASG acting on $D_{n+1}$.
\end{corollary}

It was proved in \cite{QSMomentMapJFA,QVUnitBall1,QVUnitBall2} the mutual commutativity of Toeplitz operators with nilpotent symbols, i.e.~with $\R^{n+1}$-invariant symbols. Furthermore, these previous works ensure the existence of unitary maps with respect to which the corresponding Toeplitz operators transform into multiplier operators. We now use these previous results to obtain a second proof of the diagonalizing formulas for Toeplitz operators with $\HH_n$-invariant symbols. 

Let us fix a symbol $a \in L^\infty(D_{n+1})^{\HH_n}$ and consider a function $\widetilde{a}$ such that $a(z) = \widetilde{a}(\im(z_{n+1}) - |z'|^2)$, for almost every $z \in D_{n+1}$. Let us now define the function $f : \R^n \times \R_+ \rightarrow \C$ by the expression
\[
	f(u',t) = \widetilde{a}\bigg(\frac{1}{2t}\bigg).
\]
In particular, we also have
\[
	\widetilde{a}(t) = f\bigg(u',\frac{1}{2t}\bigg)
\]
for almost every $u' \in \R^n$ and $t \in \R_+$. From this we conclude that
\begin{align*}
	a(z) &= \widetilde{a}(\im(z_{n+1}) - |z'|^2) \\
	&= f\bigg(-\frac{2\im(z')}{\im(z_{n+1}) - |z'|^2},
		\frac{1}{2(\im(z_{n+1}) - |z'|^2)}
		\bigg),
\end{align*}
for almost every $z \in D_{n+1}$. Note that the set of the first $n$ variables of $f$, which belongs to $\R^n$, can be chosen arbitrarily since this function does not depend on it. Once this expression has been established we can apply \cite[Theorem~7.8]{QSMomentMapJFA} to conclude the existence of a unitary map $U : \cA^2_\lambda(D_{n+1}) \rightarrow L^2(\R^n) \otimes L^2(\R_+)$, independent of the symbol under consideration, such that we have $U T^{(\lambda)}_a U^* = M_{\widehat{\gamma}_{f,\lambda}}$ where the function $\widehat{\gamma}_{f,\lambda}$ is given~by
\begin{align*}
	\widehat{\gamma}_{f,\lambda}&(y',\xi) = \\
	&=\frac{\xi^{\lambda + \frac{n}{2} + 1}}%
		{2^n \pi^{\frac{n}{2}} \Gamma(\lambda + 1)}
		\int_{\R^n \times \R_+}
		\frac{f(u',t) 
			e^{-\frac{\xi}{t} 
				-\sVert[1] \frac{\sqrt{\xi} u'}{2t} 
				- y' \sVert[1]^2}}%
		{t^{\lambda + n + 2}} \dif u' \dif t.
\end{align*}
Next, we replace our choice of $f$ in terms of $\widetilde{a}$ to obtain after some computations the following
\begin{align*}
	\widehat{\gamma}_{f,\lambda}&(y',\xi) = \\
		&=\frac{\xi^{\lambda + \frac{n}{2} + 1}}%
		{2^n \pi^{\frac{n}{2}} \Gamma(\lambda + 1)}
		\int_{\R^n \times \R_+}
		\frac{\widetilde{a}(1/(2t)) 
			e^{-\frac{\xi}{t} 
				-\sVert[1] \frac{\sqrt{\xi} u'}{2t} 
				- y' \sVert[1]^2}}%
		{t^{\lambda + n + 2}} \dif u' \dif t \\
		&=\frac{\xi^{\lambda + \frac{n}{2} + 1}}%
		{2^n \pi^{\frac{n}{2}} \Gamma(\lambda + 1)}
		\int_0^\infty\frac{\widetilde{a}(1/(2t))
			e^{-\frac{\xi}{t}}}%
		{t^{\lambda + n + 2}}
		\bigg(
		\int_{\R^n} e^{-\sVert[1] 
			\frac{\sqrt{\xi} u'}{2t} 
			- y' \sVert[1]^2} \dif u'
		\bigg) \dif t \\
		&=\frac{\xi^{\lambda + \frac{n}{2} + 1}}%
		{2^n \pi^{\frac{n}{2}} \Gamma(\lambda + 1)}
		\int_0^\infty\frac{\widetilde{a}(1/(2t))
			e^{-\frac{\xi}{t}}}%
		{t^{\lambda + n + 2}}
		\bigg(
		\int_{\R^n} e^{-\sVert[1] 
			\frac{\sqrt{\xi} u'}{2t} 
			\sVert[1]^2} \dif u'
		\bigg) \dif t \\
		&=\frac{\xi^{\lambda + \frac{n}{2} + 1}}%
		{2^n \pi^{\frac{n}{2}} \Gamma(\lambda + 1)}
		\int_0^\infty\frac{\widetilde{a}(1/(2t))
			e^{-\frac{\xi}{t}}}%
		{t^{\lambda + n + 2}}
		\bigg(
		\int_{\R^n} e^{-\frac{\xi}{4t^2} |u'|^2} 
		\dif u'
		\bigg) \dif t \\
		&=\frac{\xi^{\lambda + \frac{n}{2} + 1}}%
		{2^n \pi^{\frac{n}{2}} \Gamma(\lambda + 1)}
		\int_0^\infty\frac{\widetilde{a}(1/(2t))
			e^{-\frac{\xi}{t}}}%
		{t^{\lambda + n + 2}}
		\bigg(
		\frac{\pi^\frac{n}{2} (2t)^n}{\xi^\frac{n}{2}}
		\bigg) \dif t \\
		&=\frac{\xi^{\lambda + 1}}%
		{\Gamma(\lambda + 1)}
		\int_0^\infty\frac{\widetilde{a}(1/(2t))
			e^{-\frac{\xi}{t}}}%
		{t^{\lambda + 2}}
		\dif t
		=\frac{\xi^{\lambda + 1}}%
		{\Gamma(\lambda + 1)}
		\int_0^\infty \widetilde{a}\bigg(\frac{s}{2}\bigg) e^{-\xi s}
		s^{\lambda}
		\dif s.
\end{align*}
We note that this function is independent of $y' \in \R^n$, and so once we rename it to $\gamma_{\widetilde{a},\lambda}$ to emphasize its dependence on $\widetilde{a}$ we obtain the same expression given by equation~\eqref{eq:gamma-for-Heisenberg}. 

\begin{remark}\label{rmk:Heisenberg-vs-nilpotentMASG}
	Although the diagonalizing formulas from Theorem~\ref{thm:Heisenberg-Toeplitz-multiplier} can be obtained from the corresponding ones for the nilpotent MASG acting on $D_{n+1}$, there are a number of advantages using our techniques.
	\begin{itemize}
		\item Our computations using the Heisenberg group are more direct and short, even after omitting the computations needed to obtain the result for the nilpotent MASG.
		\item To translate the Cauchy-Riemann equations on $D_{n+1}$ to corresponding ones after changes of coordinates using the nilpotent MASG as found in \cite{QSMomentMapJFA,QVUnitBall1,QVUnitBall2} required to apply Fourier transform on $n+1$ (real) variables. It is remarkable that our method using the Heisenberg group required to apply Fourier transform only over the (last) real variable of $\HH_n$: in some sense, the rest of Cauchy-Riemann equations simply went through since we obtained Fock spaces.
		\item The use of the Heisenberg group and the group-moment coordinates provided a unitary map $U_\lambda$ that uncovers an important fact: the weigthed Bergman spaces on $D_{n+1}$ are the direct integral of the family of all Fock spaces
		\[
			\cA^2_\lambda(D_{n+1}) \simeq 
				\int_{\R_+}^\oplus \cF^2_{2\xi}(\C^n) \dif \xi.
		\]
		This fact was already known (see for example \cite{SNVNilpotent}). However, this work unveils the source of this unitary equivalence: the action of the Heisenberg group on $D_{n+1}$. This is actually related to the fact that, as noted in the previous item, the first $n$ Cauchy-Riemann equations went through.
	\end{itemize}
	In conclusion, the use of group-moment coordinates associated to $\HH_n$ yield interesting formulas and facts while explaining their source.
\end{remark}

\section*{Statements and Declarations}
\textbf{Funding }This research was partially supported by a Conahcyt scholarship held by the first author, by SNI-Conahcyt and by Conahcyt Grants 280732 and 61517. \\
\textbf{Competing interests} The authors have no competing interests to declare that are relevant to the content of this article. \\
\textbf{Data Availability} Data sharing not applicable to this article as no datasets were generated or analyzed during the current study.

\end{document}